\documentclass[12pt]{amsart}
\usepackage{amsmath,amssymb,enumerate}
\usepackage{epsf,epsfig,amsfonts,graphicx,a4wide}  
 \usepackage[applemac]{inputenc}
\usepackage[dvipsnames]{xcolor}

\usepackage{mathdots}

\usepackage{url}
\usepackage{amsmath,amssymb,amsthm}

\usepackage{amssymb}
\usepackage{amsthm}

\usepackage{soul}

\newtheorem{definition}{Definition}[section]
\newtheorem{Theorem}{Theorem}[section]

\newtheorem{Proposition}{Proposition}[section]

\newtheorem{Problem}{Problem}[section]

\newtheorem{Definition}{Definition}[section]
\newtheorem{Ex}{Example}[section]
\newtheorem{Remark}{Remark}[section]

\theoremstyle{remark}

\newcommand{\be}{\begin{equation}}
\newcommand{\ee}{\end{equation}}

\newcommand{\R}{\mathbb{R}}\newcommand{\Id}{\textrm{\rm Id}}

\newcommand{\gl}{\mathrm{gl}}

\newcommand{\ddd}{\mathrm{d}}

\newcommand{\dd}{{\mathrm d}\,}

\newcommand{\tr}{\operatorname{tr}}

\newcommand{\trace}{\operatorname{tr}}

\newcommand{\weg}[1]{}

\usepackage{biblatex} %Imports biblatex package
\title[Research problems on
Nijenhuis   geometry and integrable systems]{Research problems on relations between Nijenhuis   geometry and integrable systems}
\author{Alexey V. Bolsinov}\address{A.B.: School of Mathematics,
 Loughborough University,
 LE11 3TU, UK }\email{\tt A.Bolsinov@lboro.ac.uk}  

\author{Andrey Yu. Konyaev}\address{A.K.: Faculty of Mechanics and Mathematics, Moscow State University, 119992, Moscow, Russia}
 \email{\tt  maodzund@yandex.ru}
 \author{Vladimir S. Matveev}\address{V.M.:
Institut f\"ur Mathematik, Friedrich Schiller Universit\"at Jena,
07737 Jena, Germany} \email{\tt  vladimir.matveev@uni-jena.de} 
  
\date{}

\begin{document}

\begin{abstract} The paper surveys open problems and questions
related to interplay between the theory of integrable systems with infinitely and 
finitely many degrees of freedom and Nijenhuis geometry.
This text has grown out  from  preparatory materials  for the series of research symposia and workshops on Nijenhuis geometry and integrable systems held at SMRI (Sydney)  and  MATRIX (Creswick) in February 2022  and at La Trobe University (Melbourne) and MATRIX in February 2024, and from the open problem sessions at these events.  
It  includes both relatively simple questions to get familiar with the topic, as well as challenging problems that are of great importance for the field. 

\vspace{2ex}

\noindent {\bf MSC Class:} 37K06; 37K10; 37K25; 37K50; 53B10; 53A20; 53B20; 53B30; 53B50; 53B99; 53D17; 53D20; 53D22; 37J06; 37J11; 37J35; 70H06. 
\end{abstract}

\maketitle
\tableofcontents

\section{Basics of Nijenhuis Geometry} \label{Intro}

Basic facts in this area are systematically presented in the series of four papers
\cite{nij1, nij2, nij3, nij4} (to be, perhaps, continued).   The paper \cite{nij1} is introductory and, we hope, reader-friendly.  In addition to well known classical facts,  it introduces several new concepts as well as some new terminology (e.g. related to singularities of Nijenhuis operators).  It also contains several new (sometimes quite surprising) results.  Below we briefly recall some useful info following \cite{nij1}.  

\subsection{What does Nijenhuis geometry study?}
The main object in Nijenhuis Geometry are  $(1,1)$-tensor fields. We also call   them {\it operator fields} or simply {\it operators}.    

\begin{Definition}\label{def:nijop}
{\rm
A \emph{Nijenhuis operator} is a $(1,1)$-tensor $L$ with vanishing Nijenhuis torsion $\mathcal{N}_L$, that is, 
$$
 \mathcal{N}_L(v, w) \overset{\mathrm{def}}{=} L^2 [v, w] + [Lv, Lw] - L [Lv, w] - L [v, Lw] =0 
$$
for arbitrary vector fields $v$ and $w$. 

A manifold $M$ endowed with such an operator is called a {\it Nijenhuis manifold}.   
}\end{Definition} 

Notice that vanishing of the Nijenhuis torsion is the simplest geometric condition on a $(1,1)$-tensor, and this is the reason why Nijenhuis operators naturally appear in many areas of geometry, mathematical physics and even algebra.

The ultimate goal of Nijenhuis Geometry  is to answer the following natural  questions:  

\begin{enumerate} 

\item Local description: to what form can one bring a Nijenhuis operator near almost every point by a local coordinate change?  \label{q:1} 

\item Singular points: what does it mean for a point to be generic or singular in the context of Nijenhuis  geometry?
What singularities are non-degenerate? What singularities are stable?  How do Nijenhuis operators behave near non-degenerate and stable singular points?  
 
 \item Global properties: what restrictions on 
a Nijenhuis operator are imposed by compactness of the manifold?  And conversely,  what are topological obstructions to a Nijenhuis manifold carrying a Nijenhuis operator with specific properties (e.g. with no singular points)?    \label{q:3}

\end{enumerate}

 The relation of Nijenhuis geometry to geometric  theory of finite- and infinite-dimensional  integrable systems was observed independently in different setups, and also was touched in the series of ``Application of Nijenhuis Geometry'' papers \cite{nijapp1, nijapp2, nijapp3, nijapp4, nijappl5}.   In the next subsections \S\S\ref{sec:sing}, \ref{sec:benenti}, we list some  already obtained results related to different aspects of Nijenhuis geometry which are expected to be useful in  the theory of integrable systems.    The general philosophy behind many open problems/questions in this paper is to apply already obtained results on the above general goals (\ref{q:1}-\ref{q:3}) in the theory of integrable systems. In the next subsections we give an overview of results that might be interested for applications.

\subsection{Singular points in Nijenhuis geometry} \label{sec:sing}

Let $L$ be a Nijenhuis operator on a smooth manifold $M$.  At each point $p\in M$,  we can define the {\it algebraic type}  (or {\it Segre characteristic}) of  $L(p):  T_pM \to T_pM$  as the structure of its Jordan canonical form which is characterised by the sizes of Jordan blocks related to each eigenvalue $\lambda_i$ of $L(p)$ (the specific values of $\lambda_i$'s are not important here).

\begin{Definition}\label{def:alggen}
{\rm
A point $p\in M$ is called {\it algebraically generic}, if the algebraic type of $L$ does not change in some neighbourhood $U(p)\subset M$.  In such a situation, we will also say that $L$ is algebraically generic at $p\in M$.
}\end{Definition}

\begin{Definition} \label{def:singpoint}
{\rm
 A point  $p\in M$ is called {\it singular}, if it is not algebraically generic.   
}\end{Definition}

Note that the manifold and the Nijenhuis operator are always  assumed to be as smooth as necessary, say  $C^\infty$-smooth or even analytic.  The word ``singular'' is used as an antonym to ``generic'', and the meaning of  algebraically generic is above:   the sizes of Jordan blocks do not change in a small neighborhood.  We also always assume that the manifold is connected.

It follows immediately from the continuity of $L$ that algebraically generic points form an open everywhere dense subset $M_\circ\subseteq M$.

Let us emphasise that in the context of our paper,  ``singular'' is a geometric property and refers to a bifurcation of the algebraic type of $L$  at a given point $p\in M$.   In algebra,  the property of ``being singular or regular''  is usually understood in the context of the representation theory.

\begin{Definition}
\label{def:algregular}
{\rm
An operator  $L(p)$   (and the corresponding point $p\in M$) is called  {\it $\mathrm{gl}$-regular}, if its $GL(n)$-orbit $\mathcal O(L(p))=\{ XL(p)X^{-1}~|~ X\in GL(n)\}$ has maximal dimension, namely, $\dim \mathcal O(L(p)) = n^2- n $.  Equivalently,  this means that the geometric multiplicity of each eigenvalue  $\lambda_i$ of $L(p)$ (i.e., the number of Jordan $\lambda_i$-blocks in the canonical Jordan decomposition over $\mathbb C$) equals one. }\end{Definition}

One more notion is related to analytic properties of the coefficients of the characteristic polynomial $\chi_{L(x)}(t) = \det \bigl(t\cdot\Id - L(x)\bigr)= \sum_{k=0}^n t^k \sigma_{n-k}(x)$ of $L$.

\begin{Definition}
\label{def:nondeg}
{\rm
A point $p\in M$ is called  {\it differentially non-degenerate}, if the differentials $\ddd\sigma_1(x),\dots,\ddd\sigma_n(x)$ of the coefficients of the characteristic polynomial $\chi_{L(x)}(t)$  are linearly independent at this point.  
}\end{Definition}

The notions of algebraic genericity, algebraic regularity and differential non-degeneracy are related  to each other.  For example, differential non-degeneracy automatically implies $\gl$-regularity \cite{nij1}.  On the other hand, 
if $L(p)$ is $\mathrm{gl}$-regular and, in addition,  diagonalisable, then $p$ is algebraically generic (i.e., non-singular).   Conversely, if we assume that  $L(p)$ is not $\mathrm{gl}$-regular, we may expect that $p$ is singular in the sense of Definition \ref{def:singpoint}.  This will definitely be the case if $L$ is real analytic and there are points $x\in M$ at which $L(x)$ is $\mathrm{gl}$-regular.   

Notice that in the context of Definition \ref{def:algregular},  the `most non-regular'  operators are those of the form $L(p) = \lambda\cdot \Id$, $\lambda\in \R$, as in this case $\dim \mathcal O(L) = 0$, i.e., $L$ is a fixed point w.r.t. the adjoint action.   It is known that for any function  $\lambda$ the $(1,1)$-tensor field  $\lambda \cdot \Id $ is a   Nijenhuis operator. Those points where $L(p)$ becomes scalar play a special role in Nijenhuis geometry  (see \cite{nij2}) and we distinguish this type of points by introducing

\begin{Definition}
\label{def:scalar}
{\rm
A point $p\in M$ is called {\it of scalar type}, if $L(p) = \lambda\cdot \Id$, $\lambda\in \R$.
}\end{Definition}  

In the real analytic case,  scalar type points are automatically singular unless  $L$ is scalar on the whole manifold $M$, i.e.,  $L(x) = \lambda(x) \cdot \Id$  for a certain global function $\lambda : M \to \R$. 

\subsection{Useful properties of a Nijenhuis operator}

In this section, $L$ denotes a Nijenhuis operator.

\begin{itemize}
    
\item  For any vector field $\xi$ we have
$$
\mathcal{ L}_{L\xi}  (\det L)  =
\det L \cdot \mathcal{ L}_\xi  \tr L
$$
or, equivalently,    
\begin{equation}
\label{eq:lndetL}
L^* \ddd  (\det L)=\det L \cdot \ddd \tr L.
\end{equation}
More generally,  the differential of the characteristic polynomial $\chi(t)=\det (t\cdot \Id-L)$ (viewed as a smooth function on $M$ with $t$ as a formal parameter) satisfies the following relation:
\begin{equation}
\label{eq10}
L^* \bigl(\ddd \chi(t) \bigr) - t \cdot \ddd \chi(t)=\chi(t) \cdot \ddd\tr L.
\end{equation}

\item Let $\sigma_1, \dots, \sigma_n$ be the coefficients of the characteristic polynomial 
$$
\chi (t)= \det (t\cdot \Id - L) = t^n + \sigma_1 \, t^{n-1} + \sigma_2 \, t^{n-2} + \dots + \sigma_{n-1} \, t + \sigma_n 
$$
of a Nijenhuis operator $L$.   Then in any local coordinate system $x_1,\dots, x_n$ the following matrix relation holds:
\begin{equation}
\label{eq:Lexplicit}
J(x) \, L (x)= S_\chi (x)  \, J(x),  \quad\mbox{where }   S_\chi (x) = \begin{pmatrix}
-\sigma_1(x)  &  1  &    &    & \\ 
-\sigma_2(x) & 0 & 1&  & \\  
\vdots   &  \vdots &\ddots &\ddots &  \\
-\sigma_{n-1}(x)  & 0 &\dots & 0 &1 \\
-\sigma_n(x) & 0 & \dots & 0 & 0
\end{pmatrix}
\end{equation}
and $J(x)$ is the Jacobi matrix of the collection of functions $\sigma_1,\dots, \sigma_n$ w.r.t. the variables  $x_1,\dots, x_n$.

\item  If $L$ is differentially non-degenerate then there exists a local coordinate system $x_1,\dots,x_n$ such that
$$
L = \begin{pmatrix}
-x_1  &  1  &    &    & \\ 
-x_2 & 0 & 1&  & \\  
\vdots   &  \vdots &\ddots &\ddots &  \\
-x_{n-1}  & 0 &\dots & 0 &1 \\
-x_n & 0 & \dots & 0 & 0
\end{pmatrix}.
$$
We recall that the characteristic polynomial of such $L$ is 
$$
t^n + x_1 t^{n-1}+ x_2 t^{n-2}+\cdots + x_n,
$$
so the functions $\sigma_i$ from \eqref{eq:Lexplicit} are just the coordinates.

\item if we have two Nijenhuis operators $L_1$ and $L_2$ whose characteristic polynomials on $M$ coincide and  the coefficients of these polynomials  are functionally independent almost everywhere on $M$, then $L_1=L_2$.    

\item (Haantjes Theorem, see e.g. \cite{Haantjes}). If $L$ is diagonalisable over $\mathbb R$ with distinct eigenvalues then there is a local coordinate system $y_1,\dots, y_n$ such that 
$$
L=\operatorname{diag}(f_1(y_1), f_2(y_2), \dots, f_n(y_n).
$$
\item Let $L$  be a Nienhuis operator. Then $f(L)$ is a Nijenhuis operator for any  real analytic matrix function, see e.g. \cite{splitting1} for a discussion of functions of
operators.

\item Splitting Theorem \cite{splitting1, nij1}.  
Consider a point $p\in M$ and let $\lambda_1, \dots, \lambda_n$ be the eigenvalues (possibly complex) of $L(p)$ counted with multiplicities.   Assume that they are divided into two groups $\{\lambda_1,\dots,\lambda_r\}$ and $\{\lambda_{r+1},\dots,\lambda_n\}$ in such a way that equal eigenvalues as well as complex conjugate pairs belong to the same group, i.e., for every $i\in \{1,\dots,r\}$, $j\in \{r+1,\dots,n\}$ we have $\lambda_i\not=\lambda_j$ and $\lambda_i\not=\bar\lambda_j$.  
Then the characteristic polynomial of $L$ can be factorised
$\chi_L(t)=\chi_1(t)\, \chi_2(t)$ into two real polynomials with no common roots. Moreover, this can be done in some neighborhood of $p\in M$. 

The Splitting Theorem states that there exists a coordinate system $(x_1,\ldots, x_r,y_{r+1}, \ldots, y_{n})$ adapted to this factorisation in which the operator $L$ takes the block-diagonal form
 \begin{equation}
 \label{matl} 
L(x,y)=\begin{pmatrix} L_1(x) & 0 \\  0 & L_2(y)\end{pmatrix},
\end{equation}
where the entries of the $r\times r$ matrix $L_1$ depend on the $x$-variables only, and the entries of the $(n-r)\times (n-r)$ matrix $L_2$ depend on the $y$-variables only.     In other words,
$L$ splits into a direct sum of two Nijenhuis operators: $L(x,y) = L_1(x) \oplus L_2(y)$, with $\chi_i(t)$ being the characteristic polynomial of $L_i$.  
\end{itemize}

Papers \cite{nij2}, \cite{nij3}, \cite{nij4} are devoted respectively to
\begin{itemize}
    \item {\it Linearisation of Nijenhuis operators at singular points and left symmetric algebras}. Though the study  of singular points is an important direction  in Nijenhuis geometry, as well as in the research program of our workshops, the   relation to left symmetric algebras is not explicitly mentioned in the subprojects below. See \cite{KonyaevII,Serena} for recent results on this topic. 
    
    \item  {\it $\gl$-regular Nijenhuis operators}. The paper \cite{nij3} is an important step for the development of Nijenhuis geometry program, particularly in view  of possible applications.  The results of this paper were developed and applied in   \cite{nijappl5,nij4}, which focus more on the relation between $\gl$-regular Nijenhuis operators,  integrable systems and some questions in differential geometry. 
    
    \item {\it Symmetries and conservation laws of Nijenhuis operators}. This paper is very important for the programme, in particular for the F- and Frobenius structures, see \S \ref{sec:3} and also for  the study of BKM systems, see \S \ref{sec:1}. 
\end{itemize}

The two last topics are closely related to the problems discussed below. 
Necessary  definitions and facts are presented in the relevant sections.  

Scalar type points and $\gl$-regular points are, informally speaking,  boundary types of singular points. The study of  intermideate types of singularities, 
when the singular point is neither $\gl$-regular nor of scalar type, is an interesting topic, see e.g. \cite{Akpan1,Akpan2, Akpan3}. 

An obligatory  Nijenhuis geometry reading  list  definitely includes \cite{nij1}. Those interested in F- and Frobenius structures should also study  \cite{nij4}. BKM systems are introduced in \cite{nijapp4}, and further developed in \cite{BKMreduction}. For geometric Poisson brackets,  \cite{nijapp3} may be necessary. Separation of variables was handled in \cite{separation}.

\subsection{What is a  Benenti system?} \label{sec:benenti}

The terminology was suggested in \cite{BM2003,IMM}; it  is slightly misleading as S. Benenti studied a  special case only. 
We start with a pair $(\textrm{metric $g$}, \ \textrm{$g$-selfadjoint (1,1)-tensor $L$})$ satisfying the following condition:\footnote{Equivalent form of this condition is \eqref{eq:B1}.}. 
\begin{equation} \label{eq:M0}
    L_{ij, k} = \lambda_{j} g_{ij}+ \lambda_{i} g_{jk}.
\end{equation}
In the above equation, we use $g$ for manipulation with indices and  denote by comma the covariant derivative with respect to the Levi-Civita connection of  $g$. The 1-form $\lambda = (\lambda_i)$ staying on the right hand side is necessarily $\lambda = \tfrac{1}{2}\, \ddd\trace(L)$.

The equation \eqref{eq:M0} appeared many times independently in different branches of mathematics. In particular, in the context of integrable systems, the solutions $L$  of  \eqref{eq:M0}  (for a given $g$) are  called {\it  special conformal Killing tensors} in \cite{Crampin}. In differential geometry, \eqref{eq:M0}  appeared at least in \cite{Sinjukov} in the context of geodesically equivalent metrics. We will give more details in \S \ref{sec:conservation}; let us note that the equation \eqref{eq:B1} from that section is equivalent to \eqref{eq:M0}.

Given $L^i_j$ satisfying \eqref{eq:M0}, one can construct a family of commutative integrals for the geodesic flow of $g$. Namely, consider the following family of functions on the tangent bundle $TM$ polynomially depending on $t$ as a parameter:
\begin{equation} \label{eq:MI}
   \xi\in TM \ \mapsto \  I_t(\xi)= \det(t\,\Id{-} L) g((t\,\Id {-}L)^{-1} \xi, \xi)= I_{0}(\xi)t^{n-1} {+} I_{1}(\xi)t^{n-2}{+} \cdots {+} I_{n-1}(\xi).  
\end{equation}
 
\begin{Theorem}[\cite{MT}] \label{thm:MT}
The functions $I_i$ commute with respect to  the pullback of the canonical Poisson structure on $T^*M$ with respect to the  standard  identification of $TM$ and $T^*M$ by $g$. 
\end{Theorem}

In other words, for any $t$, the (0,2)-tensor given by \eqref{eq:MI} is a Killing tensor. Note that the integral $I_0$ is twice the Hamiltonian $g(\xi, \xi)$ of the geodesic flow.

One can show that the number of  functionally independent integrals $I_i$ is the degree of the minimal polynomial of $L$ at a generic point, see \cite[Lemma 5.6 and Corollary 5.7]{Gover} or \cite[Theorem 2 and Proposition 3]{Topalov}. 

The equation \eqref{eq:B1} is covariant, so one can study it in any coordinate system. The diagonal coordinates for $L$ correspond to separation of variables, they are discussed in \S \ref{sec:2.4.1}, see formula \eqref{eq:M2} in which the kinetic part corresponds to the integrals from Theorem \ref{thm:MT}.  As discussed in  \S \ref{sec:2.4.1}, the first companion coordinates may  be useful for construction of certain solutions of BKM systems, while the
second companion coordinates will be used in topics from \S \ref{sec:conservation}, in the interplay of conservation laws and geodesically compatible metrics.

\section{Study of BKM systems}  \label{sec:1}

\subsection{What are BKM systems?}\label{sec:1.1}
A large family of integrable PDE systems was constructed in \cite{nijapp4}.
There are 4 types of systems, I--IV.  It is easier to explain the  systems BKM II and BKM IV, so we start with them.  
These are systems of the form 
\begin{equation} \label{eq:M1}
\tfrac{\partial u}{\partial t}= \xi[u], 
\end{equation}
where $u=(u^1(x,t),...,u^n(x,t))$ are unknown functions of two variables and the components of   $\xi[u]=(\xi^1[u],...,\xi^n[u])$ are    certain  polynomials   in $u_x= \tfrac{\partial u}{\partial x}$, $u_{xx}=\tfrac{\partial^2 u}{\partial x^2}$ , $u_{xxx}=\tfrac{\partial^3 u}{\partial x^3}$ with coefficients depending on $u$.

The systems of this type are called {\it dynamical or evolutionary (PDE) systems}.  The form of \eqref{eq:M1} is visually similar to the standard way of writing a 
 usual finite-dimensional dynamical ODE system on an 
unknown vector-function
$u=(u^1(t),...,u^n(t))$: 
\begin{equation} \label{eq:ana} \dot u(t)= X(u(t))\ , \ \ \textrm{ where    $X(u)$ is a vector field}.\end{equation} 
Let us discuss further analogies between \eqref{eq:M1} and \eqref{eq:ana} which may be useful to have in mind.

The initial condition for \eqref{eq:M1} is a curve $x\mapsto u(x, 0)$; one can treat \eqref{eq:M1} as a dynamical system on the space    of curves\footnote{ In the real-analytic category, jet of a curve at one point determines the curve. So the space of curves is essentially the jet bundle.}.  
The tuple of differential polynomials  $\xi$ can be viewed as  an infinite-dimensional vector field. If all the objects are real-analytic, the Cauchy-Kovalevskaya Theorem  guarantees the existence  and uniqueness of local solutions  with a given initial condition  $u(x, 0). $ Examples show that some solutions ``explode'', i.e.,  go to infinity in a finite time.

A BKM system  depends on the choice of the following parameters: 

\begin{enumerate}
\item  dimension $n$;
\item  a differentially nondegenerate\footnote{An operator is differentially nondegenerate, if the coefficients of its characteristic polynomial are functionally independent, see Definition \ref{def:nondeg}. Locally isomorphic 
examples of differentially nondegenerate operators are given in \eqref{dnd}.} Nijnehuis operator $L$ in dimension $n$;
\item a polynomial  $m(t)$ of degree at most $n$ in BKM II and of degree at most $n-1$ in BKM IV; 
\item  a number $\lambda\in\mathbb R$ (or $\mathbb C$) in BKM I and  III  and  
a root of $m(t)$ in BKM II.  
\end{enumerate}

\begin{Remark} {\rm The systems constructed in \cite[\S 2]{nijapp4} are  slightly more general, in particular one can in addition choose  an integer number $N\ge  0$ and natural numbers  $n_0, n_1, \dots  , n_N$ and $\ell_1,\dots , \ell_N$ satisfying certain conditions.  In the present paper, we restrict    the consideration to, and will speak only about  the systems such that $N=0$. In this case, the conditions on $n_i$ and $\ell_i$ imply  $n_0=n$. }
\end{Remark}

Each such choice gives, by explicit formulas explained in \cite[\S 2]{nijapp4}, 
an integrable  system of the form \eqref{eq:M1}.

The BKM systems of type I and III  can also be written in the form \eqref{eq:M1}, with the following difference: the 
``vector field'' $\xi$ is now a differential series  in $u_x, u_{xx}=u_{x^2}, u_{xxx}=u_{x^3}, u_{xxxx}=u_{x^4},\dots$  The series has a certain special form which allows one to rewrite it as a system of PDEs, that is, 
in a  form such that only finite (in fact, at most third) order derivatives are used. More precisely, one introduces  an additional auxiliary  unknown function  $q(x,t)$ and  
rewrites the initial system  as a system of third order PDEs on the unknown functions $u,q$; the higher derivatives of  $u$ are then hidden in the PDEs involving  $q$ which is called  a {\it differential constraint}.  See \cite[\S 2]{nijapp4} for details.

We suggest to view the BKM systems  as follows:   we work on a manifold with a local coordinate system $u=(u^1,\dots,u^n)$.  The background structure on the manifold is a differentially nondegenerate Nijenhuis operator. 
For any choice of parameters specified above, we obtain a dynamical  system on the space of curves.

BKM systems are integrable in many senses\footnote{The notion of integrability of infinite dimensional systems is not completely formalised in literature, see e.g. \cite{Deift2019}.}. 
In particular, each BKM system  has  an infinite hierarchy of conservation laws  (finite-dimensional analog of a  conservation  law is an integral of a dynamical system), and an  infinite hierarchy of symmetries  (a finite-dimensional analog of a symmetry for $\dot u= X(u)$ is a vector field commuting with $X$).

Though it is possibly not clear from the proofs of \cite{nijapp4} and only tangentially mentioned there, the BKM systems are  multi-Hamiltonian. The corresponding pencils of compatible Poisson structures were 
constructed in \cite{nijapp3}.  Note that these pencils are related to separation of variables for metrics of constant curvature, as discussed in \cite{nijapp3,separation}. Namely,   for any  separating coordinate  system for a metric of constant curvature there exists a family of compatible Poisson structures, and vice versa.

 Depending on the choice of parameters and a suitable coordinate system\footnote{Differentially nondegenerate Nijenhuis operator is essentially unique, but choosing different coordinate systems gives  us visually different, though  equivalent via a local diffeomorphism $u_{\mathrm{new}}=u_{\mathrm{new}}(u_{\mathrm{old}})$,  systems.}, special cases of BKM systems give us a series of interesting examples. In  particular, 
KdV,  Harry Dym,  Camassa-Holm, 
Dullin-Gottwald-Holm are BKM systems with $n=1$ and  Kaup-Boussinesq systems,  coupled
Harry Dym, coupled KdV and multi-component Camassa-Holm  are BKM systems with $n\ge 2$. The form of the Nijenhuis operator and the choice of parameters for these systems  are specified in  \cite[\S 3.3]{nijapp4}.

\subsection{Research direction 1: Looking for other integrable systems which are   BKM systems }

As discussed in \S  \ref{sec:1.1}, many integrable systems are BKM systems. There is almost an industry in mathematical physics of searching for new integrable systems.   Many  integrable systems appeared in studying problems of applied interest and  come directly from physical considerations.  So now many integrable systems can be found in literature.

The goal of this research direction is to understand  which  of them are BKM systems.
The point is  that we plan to adjust 
methods  developed for specific  integrable systems to the whole class of BKM system; writing an integrable system as a BKM system will allow us to use these methods for more systems. An additional motivation is that this will allow us to identify systems which coincide up to a local diffeomorphism of the unknown functions $u=(u^1,\dots, u^n)$, since in each dimension $n$, the differentially nondegenerate operators are essentially locally isomorphic but the corresponding BKM systems may be visually very different.

That is, given an $n$-component integrable  system of PDEs of the form \eqref{eq:M1} (or with differential constraints as BKM I and III) found in a literature, can  one find a differentially nondegenerate $n$-dimensional operator $L$  and parameters (1--4) from \S \ref{sec:1.1} such that the system ``transforms'' into a BKM system?

The following  three  test questions may help to understand whether an integrable system of the form \eqref{eq:M1} has a chance to be a BKM system: 

\begin{itemize}
    \item[TQ1] Is the system bi-Hamiltonian?  What is the second Poisson  structure?  (In the BKM systems,   compatible Poisson structures\footnote{We recall  the definition of geometric Poisson structures of order $k$ in  
    \S \ref{sec:6.1}.}  are of third order).  
    \item[TQ2] What do the conservation laws and symmetries look like?  (In the BKM-systems, the orders of the conservation laws are even, and of symmetries are odd)
    \item[TQ3] Do stationary\footnote{Stationary solutions are characterised by the property that all but finitely many symmetries  vanish on the system. For KdV, they are sometimes called finite-zone solutions \cite{VBM}.} solutions of the system correspond, in a natural sense, to a weakly nonlinear (=  linearly degenerate) system of a hydrodynamic type? 
\end{itemize}

Of course, most interesting systems are those which describe physical phenomena. Generally, finding physically interesting systems in the class of BKM systems is a very important task, and mathematicians may need a little help from physicists on this topic.

 A possible scenario of this research could be: 
 \begin{itemize}
\item  one finds  an integrable system  e.g. in the literature; 
\item  then, one needs  to apply the test questions TQ1, TQ2, TQ3  above, and if the answers are positive, look further\footnote{One needs  to look for a coordinate system on our $n$ dimensional manifold and for parameters (1--4) described in \S \ref{sec:1.1} such that  the corresponding BKM system  becomes the system under consideration.} at  the system.  \end{itemize}

\subsection{Research direction 2: Applying established methods from the theory of integrable systems to BKM systems  }

Theory of integrable systems is an important chapter in mathematical physics and  mathematics. Such systems were studied from different perspectives; there   is  a bunch of nontrivial methods developed in the area and coming from different branches of mathematics.  

Which  of these methods   can be applied   to all or to some of BKM systems?  How to translate such a  method to the language of Nijenhuis geometry? Can one universally apply the method to all BKM systems simultaneously?

 A possible scenario of this research could be: 
 \begin{itemize}
\item  one takes  a  method which was successfully used in the  study of a certain special case  of BKM systems (for example, 
KdV) and studies both the outcome of the method and preliminary steps;  
\item  then, one needs to understand    which auxiliary objects should be constructed to make this method work; 
\item and finally one tries to construct these objects  using tools of  Nijenhuis geometry and then  apply them to all BKM systems.
   \end{itemize}

As recalled above, many integrable systems studied by classical methods, in  particular, 
KdV, coupled KdV, Harry Dym, coupled
Harry Dym, Camassa-Holm, multi-component Camassa-Holm, 
Dullin-Gottwald-Holm and Kaup-Boussinesq systems are BKM systems.

Possible  auxiliary objects  are 
``recursion operator'', ``$R$-matrix", ``Lax representation", ``Inverse Scattering Transform'', ``zero curvature represenatation''. Of course, one of the most successful methods of studying infinite-dimensional systems is related to their ``finite-gap'' or ``stationary'' finite-dimensional reductions; this method was recently generalised for all BKM systems in \cite{BKMreduction} and in the next subsection  we discuss its applications and further  developments.

\subsection{Finite-gap type  solutions of BKM systems and separation of variables}

\subsubsection{Background} \label{sec:2.4.1}
 
Finite-gap type solutions  form  a special class of solutions for certain infinite-dimensional integrable systems, see e.g. \cite{VBM}. The general scheme is as follows: 
under an additional assumption on the solution, the integrable system is reduced to a finite-dimensional integrable system. This finite dimensional integrable system is then solved  (e.g., by means of advanced algebro-geometric methods, or by reducing it  to a system of ODEs and using numerical solvers) and its solutions  
give us solutions of the initial PDE system. This is practically the only method that allows one to find exact solutions. This method can also be effectively used to  construct approximate solutions with a high degree of accuracy.
Solitons and finite-zone solutions of KdV, see \cite{VBM},  can be constructed in this way. 

Recently, this scheme for BKM systems has been developed and understood \cite{BKMreduction}. It appears that  the corresponding finite-dimensional system  is  related to separation of variables in the flat  space of dimension $N$ (``number of gaps''), and to a certain Benenti  system on this space.

We start with a BKM system (with $n$ components and  generated by a polynomial $m(t)$, see \S \ref{sec:1.1}).   Recall that our freedom is  the choice of a polynomial $m(t)$ of degree  $\le n$. 
The  construction works with some  modification with BKM systems of all four types and for any choice of the number $\lambda$. The necessary    modification only affects the formula that transforms solutions of the  finite-dimensional system to  those of \eqref{eq:M1}.

Next, take $N$ (``number of gaps'') and  consider the following  integrable systems on $2N$-dimensional manifold  $T^*\R^N$ with the standard symplectic structure.  
The commuting functions  $I_0,\dots,I_{N-1}$   are sums of terms quadratic in momenta   (=kinetic terms) and functions on the manifold (=potentials) which are given by  explicit formulas that we describe now:

We consider the Vandermonde $N\times N$  matrix 
$$
W(x_1,\dots ,x_N)= \begin{pmatrix} x_1^{N-1} &  x_1^{N-2} & \cdots &1 \\ 
x_2^{N-1} &  x_2^{N-2} & \cdots &1 \\ 
\vdots  &    & \vdots & \vdots  \\ 
x_N^{N-1} &  x_N^{N-2} & \cdots &1   \end{pmatrix}. 
$$
Next, in the coordinates $(x,p)$ on the cotangent bundle $T^*\R^N$, we consider the functions $I_0(x, p),...,I_{N-1}(x,p)$ given by 
\begin{equation} \label{eq:M2}
  \begin{pmatrix} x_1^{N-1} &  x_1^{N-2} & \cdots &1 \\ 
x_2^{N-1} &  x_2^{N-2} & \cdots &1 \\ 
\vdots  &    & \cdots & \vdots  \\ 
x_N^{N-1} &  x_N^{N-2} & \cdots &1   \end{pmatrix} \begin{pmatrix}  I_0 \\ 
I_1 \\ 
\vdots   \\ 
I_{N-1}  \end{pmatrix}   =  \begin{pmatrix}  \tfrac{1}{2}p_1^2 + V(x_1) \\ 
 \tfrac{1}{2}p_2^2 + V(x_2) \\ 
\vdots   \\ 
\tfrac{1}{2} p_N^2 + V(x_N)  \end{pmatrix}  . 
\end{equation}
Notice that \eqref{eq:M2} is a system of linear equations on $I_0,...,I_{N-1}$; if we multiply it by the  matrix  $W^{-1}$,  we obtain an  explicit formula for $I_0,...,I_{N-1}$. 
We clearly see that  the integrals $I_0,...,I_{N-1}$ are sums of the kinetic and potential terms.

The  function $V(t)$ of one variable in \eqref{eq:M2}  is given by the formula  
\begin{equation}\label{eq:M3}
V(t)= C(t)/m(t), 
\end{equation}
where $m(t)$ is the polynomial (of degree at most $n$) chosen in  the construction of the BKM system, and  $C$ is a polynomial of degree $2N+n$. The coefficients of $C(t)$ satisfy some minor\footnote{see Example \ref{ex:M1}.} assumptions and can be chosen free otherwise. For each choice of the polynomial  $C$,  we obtain a finite dimensional  integrable system each solution of which can be transformed to a solution of 
the initial BKM system.

\begin{Ex} \label{ex:M1}{\rm
In the KdV, and generally in any
BKM IV system,  one may assume that the   polynomial $C(t)$ has the form  $C(t)=1 + c_2t^2+...+c_{2N+n}t^{2N+n}$.      In the  KdV system,  the polynomial $m(t)$ is  $m(t)= m_0 = \textrm{const}$,  as the unknown vector function $u$  has $n=1$ components, and in a BKM IV system with   $n$ components  $m$  is an arbitrary polynomial of degree $<n$.}
\end{Ex}

\begin{Remark}{\rm 
The metric corresponding to the quadratic in momenta part of the  integral $I_0$ 
  has the form 
  \begin{equation} \label{eq:M4}
      g= \sum_{i=1}^N \left( \prod_{j\ne i} (x_i - x_j) \right) dx_i^2. 
  \end{equation}  
  Note that the  metric is flat and has   splitted 
  signature. It is degenerate at the points where $x_i=x_j$ for some $i\ne j$.}  
\end{Remark}

It has been observed many times  that certain famous integrable systems are related to separation of variables.  In all the cases we found in the literature, this relationship was considered as magic (a wonderful observation with no explanation). 
A historical review on the development of finite-gap solutions in the theory of  integrable systems is  in \cite{VBM}. It seems that the relation between systems generated by Poisson-commuting function 
of the form \eqref{eq:M3} and finite-gap solutions of KdV  was first  observed at the level of exact  solutions, see discussion in  \cite{DKN,moser80,moser81, veselov}.

A recent approach to this problem  which  is definitely related to ours is presented in \cite{BM2006, BM2008, BS2023,MB2010}. The authors of these papers went in other direction:  the primary object of their investigation is   a finite-dimensional  system of the form \eqref{eq:M2} from which they  come to  integrable PDE-systems which are special cases of BKM systems.

\vspace{1ex}

In \cite[Theorems 2.1, 2.2]{BKMreduction}, certain 
 solutions of BKM systems are linked with the system generated by the commutative integrals $(I_0,...,I_{N-1})$.  Namely, for each BKM system 
 there exist  two linear combinations $F= \lambda_0 I_0 + \cdots + \lambda_{N-1}I_{N-1}$ and  $H= \mu_0 I_0 + \cdots  + \mu_{N-1}I_{N-1}$ (the coefficients $\lambda_i$ and $\mu_i$ depend on the type I--IV 
 of the BKM system and on the choice of zero of the polynomial $m(t)$)
such that every solution $(x_1(\tau,t), \dots, x_N(\tau,t))$ of the system 
\begin{equation}\label{eq:M5}
   \tfrac{ \partial }{\partial \tau }(x,p) = \mathcal X_H  \  \   \  \tfrac{ \partial }{\partial t}(x,p) =\mathcal X_F \ , \ \textrm{where $\mathcal X_F$ is the  Hamiltonian vector field of $ F$,}
\end{equation}
``produces'' a solution $u(t,\tau) = \mathcal R (x(\tau, t))$ of \eqref{eq:M2}\footnote{In this $u(t,\tau)$ we replace $\tau$ by $x$.}  by means of  an explicitly given\footnote{The mapping $\mathcal R$, for big $n$ and $N$,   is given by quite complicated formulas which can be obtain algorithmically, e.g., using Maple.} mapping $\mathcal R: \R^N(x_1,...,x_{N}) \to \R^n(u_1,...,u_n)$.

\vspace{1ex}

Note that   a more natural coordinate system  for the  system \eqref{eq:M2} is  not the ``diagonal''  $x$-coordinates
(in these coordinates the kinetic energy of each integral $I_k$ is given by a diagonal matrix; in these coordinates 
the separation of variable technique is applied straightforwardly), but the 
coordinate system $w_1,...w_N$  related to the diagonal coordinates  $x_1,...,x_N$ by the formula
\begin{equation}\label{eq:M6}
  (t- x_1)(t-x_2)\cdots (t-x_N) =  t^{N}+ w_1t^{N-1}+\cdots +w_{N-1} t + w_N
\end{equation}
That is, the functions $w_i$ are elementary symmetric polynomials of $x_i$ with appropriate signs.

The  coordinate system $w$ has the following advantage over the coordinate system $x$: the formulas $u(w)$ are defined also  at the points where $x_i = x_j$ for some $i,j\in 1,...,N$. The points are indeed important; for example the multi-soliton solution of  KdV corresponds to solutions $x(\tau, t)$ of \eqref{eq:M5} passing through such points.

\begin{Remark} \label{rem:M1} {\rm
   Nijenhuis geometry stays  behind the  transformation formula from $x$ to $w$.  Actually, the quadratic part of the 
   system generated by the commuting integrals  comes from  an $N$-dimensional  differentially nondegenerate Nijenhuis operator,  and is related to the corresponding ``Benenti'' system. 
In the coordinates $x$,  the Nijenhuis operator  has the diagonal form $M_{\mathsf{diag}}(x)$  below,  and in the coordinates $w$ the Nijenhuis operator has the  first companion form  $M_{\mathsf{comp1}}(w)$. The coordinate 
transformation $w(x)$ transforms  $M_{\mathsf{diag}}(x)$  to $M_{\mathsf{comp1}}(w)$.

   \begin{equation}\label{dnd}
  M_{\mathsf{comp1}} = \left( \begin{array}{ccccc}
     -w_1 & 1 & 0 & \dots & 0  \\
     -w_2 & 0 & 1 & \dots & 0  \\
     \vdots & & & \dots & \\
     -w_{N - 1} & 0 & 0 & \dots & 1 \\
     -w_N & 0 & 0 & \dots & 0
\end{array}\right) 
 \  \   \    ,  \  \   \   \      M_{\mathsf{diag}} = \left( \begin{array}{ccccc}
     x_1 & 0  & 0   & \dots & 0  \\
     0 &  x_2 &  0 & \dots  & 0  \\
     \vdots & & & \vdots & \\
     0 & \dots  & 0 & x_{N-1} & 0 \\
     0 &  \dots  & \dots  & 0 & x_N \end{array}\right)  \end{equation}
}\end{Remark}

\subsubsection{ Research directions  on interrelation of certain  solutions of BKM systems and separation of variables }

The relation   between integrable systems  in dimension $N$ 
and stationary solutions of BKM systems is discussed in \ref{sec:2.4.1}.  We remind that  system \eqref{eq:M2}
 can be solved by separation of variables. Actually, formula \eqref{eq:M2} is adapted to separation of variables.

\begin{Problem} Obtain  soliton-type  solutions of  known special cases of BKM systems using this approach.\end{Problem}
As mentioned above, many known integrable systems, for example KdV, coupled KdV, Harry Dym, coupled
Harry Dym, Camassa-Holm, multi-component Camassa-Holm, 
Dullin-Gottwald-Holm and Kaup-Boussinesq,  are BKM systems. What are properties of the  solutions of these systems obtained by the method of \cite{BKMreduction}? Which of them can be called finite-gap or soliton solutions\footnote{Different authors view different properties as typical properties of finite-gap and soliton solutions.}? 

A particular system and a particular class of intersting  solutions is the class of peakon solutions  of the Camassa-Holm equation. These solutions have a conic-type singularity. Is the singularity 
related to ``jumps'' from one solution of  \eqref{eq:M5}  to another under the condition that these solutions belong to the same Liouville torus?

\begin{Problem} Find exact solutions of BKM systems. \end{Problem}
The system \eqref{eq:M5} can be integrated  in quadratures, that is, one can reduce it to solving a system of algebraic equations whose components are primitive  functions 
of explicitly given closed 1-forms. One can try to go further and look for exact  solutions of BKM systems using elliptic or other special functions.  See \cite[\S 1.4 and Example 1.5]{BKMreduction}.

\begin{Problem} Find a geometric realization of the integrable system  generated by  \eqref{eq:M2}.\end{Problem} The metrics which correspond to  the quadratic parts of integrals $I_0,...,I_{N-1}$   are somehow nice 
(for example, the metric \eqref{eq:M4} is flat, generally, metrics of constant curvature are expected to appear in all the cases). Take a  flat coordinate system, or a generalised flat coordinate system\footnote{Definition is in \cite[\S 1.4]{separation}.} if the metric has constant  nonzero curvature, and find the potential and the corresponding Nijenhuis operator $M$ discussed in  Remark \ref{rem:M1} in these coordinates. 

For the KdV system, this was reported to be done already in \cite{veselov}, though we did not find detailed formulas. The booklet \cite{moser81} relates  stationary solutions of the KdV to the solutions of the Neumann systems. Namely, it is shown that the potential energy of the Neumann system evaluated along the trajectory and suitably scaled is a solution of the KdV system.  Moreover, the solutions of the KdV equations obtained by this method are related to potentials in the Hills-Schr\"odinger equation whose spectrum has finitely many bands.   See also  \cite[p. 219, eqn (2.162), pp. 247, 248]{DKN}.

\begin{Problem} What is special in the  potential generated by $V(t)= C(t)/m(t)$?\end{Problem}
The potentials  that appear have a very special form, they are  constructed by  one function $V(t)$ which is the ratio of two polynomials.  The kinetic parts of the integrals depends on $N$ only and is the same for all systems. 
What nice properties do  the  systems  have?  
As mentioned above, in  \cite{moser80,moser81}, finite-gap solutions of the 
KdV equation were  related to  Hill's   equation with spectrum consisting of finitely many intervals. Can one generalise this relations to all BKM systems?

\section{Frobenius  and $\mathrm F$-manifolds} \label{sec:3}

\subsection{F-structures}  \label{sec:Fstructure}
The definition of a Frobenius manifold is complicated. We first introduce/recall   the notion of an F-structure.  Frobenius structure is a pair (F-structure, 1-form $\alpha$)  satisfying some compatibility condition as explained below in \S \ref{sec:Frobenius}.

We  start with a manifold $\mathsf M^n$ equipped with a triple: a vector field $e$, a vector field $E$ and a tensor of type $(1, 2)$ denoted by $a$. The tensor $a$ induces a natural bilinear operation on vector fields $\xi, \eta$ as
$$
\xi \circ \eta = a(\xi, \eta)
$$
The definition of an $\mathrm F$-manifold is given in terms of $e, E$ and the  operation $\circ$: we say that $\left(\mathsf M^n, \circ, e, E\right)$ is an {\it $F$-manifold} if $\circ, e, E$ satisfy the following conditions:
\begin{enumerate}
    \item $\circ$ defines a structure of commutative associative algebra on vector fields,
    \item $e$ is the unity of this algebra $\circ$,
    \item For any vector fields $\xi, \eta, \zeta, \theta$ we have
    \begin{equation}\label{hertlingmanincond}
    \begin{aligned}
    0 & = [\xi \circ \eta, \zeta \circ \theta] - [\zeta, \xi \circ \eta] \circ \theta - \zeta \circ [\xi \circ \eta, \theta] - \xi \circ [\eta, \zeta \circ \theta] + \xi \circ [\eta, \zeta] \circ \theta + \\
    & + \xi \circ \zeta \circ [\eta, \theta] - \eta \circ [\xi, \zeta \circ \theta] + \eta \circ [\xi, \zeta] \circ \theta + \eta \circ \zeta \circ [\xi, \theta].
    \end{aligned}    
    \end{equation}
    Here the square brackets $\left[\cdot, \cdot \right]$ define the standard commutator of vector fields. 
    \item For any pair of vector fields $\xi, \eta$
    \begin{equation}\label{eulervectorfieldcond}
        [E, \xi \circ \eta] - [E, \xi]\circ \eta - \xi \circ [E, \eta] = \xi \circ \eta.
    \end{equation}
    The vector field $E$ is called an {\it Euler vector field}.
\end{enumerate}
Though it  is not obvious,    formula \eqref{hertlingmanincond} defines a tensor field of type $(1, 4)$. The formula itself is related to the so-called Ako-Yano bracket (see \cite{aky} and discussion in \cite{magri}), which plays an important role in the theory of quasilinear integrable systems. Condition \eqref{eulervectorfieldcond} can be rewritten as $\mathcal L_E a = a$, where $\mathcal L_E$ stands for the Lie derivative along a vector field $E$.

    F-structure  was introduced by C.  Hertling and Yu. Manin in \cite{hm, h1}, as a pre-structure for  the (Dubrovin)-Frobenius structure. 
    $\mathrm F$-manifolds found an application in many areas of mathematics and mathematical physics, see \cite{l, h2, dav1, lpr} for overview and further references.

The relation of F-structure to Nijenhuis geometry is established in \cite[Theorem 4.4]{arsie}, where it was shown that the operator  $L$ defined by the identity
$$
L \xi = E \circ \xi,
$$
is, in fact, a Nijenhuis operator. This result shows that  $F$-manifolds are Nijenhuis manifolds. The general direction of this line of research is to use the methods and results of Nijenhuis geometry, as well as its philosophy, to obtain results on $\mathrm F$- and Frobenius structures and further in branches of mathematics and physics where Frobenius structures can be applied.

We will mostly consider the so-called  {\it regular} $F$-manifolds, which are defined as follows: the Nijenhuis operator $L:=E\circ$ is $\gl$-regular, i.e., every eigenvalue of $L$ has geometric multiplicity~1. 

As noticed above, each $\mathrm F$-manifold carries a natural Nijenhuis structure. A natural question ``What structure, in addition to a $\gl$-regular Nijenhuis operator, is necessary in order to define a regular F-structure?''  was answered in \cite{b1}:

\begin{definition}
A pair $\left(L, e\right)$, where $L$ is an operator field and $e$ is a vector field, is called  a {\it Nijenhuis operator with a unity}, if $L$ is  Nijenhuis and $e$ satisfies the identity
\begin{equation}\label{unitycond}
    \mathcal{L}_{e}L = \Id. 
\end{equation}
\end{definition} 

 Here and below $\mathcal{L}_{v}$ denotes the Lie derivative with respect to a vector field $v$. 

\begin{Theorem}[\cite{b1}] \label{t5}
The class of regular $\mathrm F$-manifolds coincides with the class of Nijenhuis manifolds with a cyclic unity.
\end{Theorem}
\begin{Remark}{\rm
   Recall that $e$ is said to be {\it cyclic} if $e, Le, \dots   ,L^{n-1}e $ are linearly independent.  This condition automatically  
   guarantees that $L$ is $\gl$-regular.}
   \end{Remark}

Theorem \ref{t5} reduces the study of  regular F-structures to a problem in Nijenhuis geometry,  as condition \eqref{unitycond} is defined in Nijenhuis terms only.

\subsubsection{Symmetries and Nijenhuis operators with unity}

As we recalled above, see Theorem \ref{t5} in \S \ref{sec:Fstructure},     regular F-structure  with cyclic $e$ are in one-to-one correspondence with $\gl$-regular Nijenhuis structures with unity. This approach to F-structures via Nijenhuis geometry appears to be very useful, as many of previously proved results on F-structures were easily reproved using  results of Nijenhuis geometry, and some explicitly stated  open problems on F-structures were solved, see e.g. \cite{b1}.

We suggest to look for further synergy  between  F-structures and Nijenhuis geometry using the concept of symmetries for Nijenhuis operators. Let us recall the definition. 
For a pair of commuting operator fields $L$ and $M$ (i.e., such that $LM = ML$) one can define a $(1, 2)$ tensor field $\langle L, M \rangle$ by
$$
\langle L, M\rangle(\xi, \eta) = M[L\xi, \eta] + L[\xi, M\eta] - [L\xi, M\eta] - LM[\xi, \eta].
$$

This definition is due to A.\,Nijenhuis \cite[formula 3.9]{nijenhuis}. Notice that $\langle L,M\rangle + \langle M,L\rangle$ coincides with the Fr\"olicher-Nijenhuis bracket of $L$ and $M$ and makes sense for all operator fields (not necessarily commuting).  

We say that $M$ is a {\it symmetry}\footnote{The word ``symmetry'' is used because the system of hydrodynamic type corresponding to $M$ commutes with that corresponding to $L$.} of an operator field $L$, if $LM = ML$ and the symmetric (in lower indices) part of $\langle L, M \rangle$ vanishes, that is, $\langle L, M \rangle (\xi, \xi) = 0$ for all vector fields $\xi$. Furthermore, we say that $M$ is  a {\it strong symmetry} of $L$, if $ML = LM$ and the entire tensor $\langle M, L \rangle$ vanishes. 

 For diagonal Nijenhuis operators with different eigenvalues, the symmetries are well understood, and their complete description is  considered as a folklore result. Description of symmetries for non-diagonalisable Nijenhuis operators is new and was obtained in \cite{nij4}, we plan to employ the results of \cite{nij4} in construction and description of F-structures.

Namely,  we want to establish a link between the   algebra of symmetries discussed in \cite{nij4}, and a certain algebraic structure. As we will  see, this idea promises to be quite fruitful. It 
gives us a natural class of commutative associative structures.  Moreover, it allows us to treat F-structures which are regular almost everywhere but not everywhere. 
Recall that in Nijenhuis geometry, those points at which the Segre characteristic of a Nijenhuis operator $L$ changes   are called {\it singular}\footnote{All objects are assumed to be smooth, the word `singular' refers to a point that is not algebraically generic, i.\,e., the Jordan decomposition structure is not constant in a small neighborhood of the point, see \S \ref{Intro}.}; the study of such points is one of the declared goals of Nijenhuis geometry, see \cite{nij1}.

In addition, one might expect the integration in quadratures to appear in this context. We also discuss it below.

\subsubsection{Link between the  algebra of symmetries discussed in \cite{nij4} and a certain algebraic structure}

Let $L$ be a $\gl$-regular Nijenhuis operator. In \cite{nij4} it was shown that
\begin{enumerate}
    \item If $M$ is a symmetry of $L$, then $M$ is a strong symmetry.
    \item If $M$ is a symmetry of $L$, then $M$ is a Nijenhuis operator.
    \item If $M, N$ are symmetries of $L$, then $MN$ is also a (strong) symmetry of $L$.
\end{enumerate}
The symmetries form a commutative associative subalgebra in the infinite dimensional algebra of operator fields. We denote it by $\mathrm{Nij}\, L$ and call {\it Nijenhuis algebra}. It is infinite dimensional, and parametrized, in analytic category, by $n$ functions of a single variable.

If $M\in \mathrm{Nij}\, L$ is a symmetry, then it can be uniquely written as 
\begin{equation}\label{decomp}
 M = g_1 L^{n - 1} + \dots + g_n \operatorname{Id}.   
\end{equation}
The coefficients $g_i$ satisfy a certain system of linear PDEs which  is integrable in quadratures using matrix-valued functions (see \cite{nij4}). In other words, all the symmetries can be constructed using only algebraic operations and integration of closed 1-forms.

We say that $M$ is a {\it regular symmetry} if locally $\ddd g_i$ are linearly independent. The existence of regular symmetries in $\mathrm{Nij}\, L$ is related to the existence of companion coordinates (see \cite{nij4}). This question is, in turn, related to the existence of solutions for some quasilinear systems.

For each symmetry $M \in \mathrm{Nij}\, L$ we can construct a tensor of type $(1, 2)$ as
$$
T_M = \ddd g_1 \otimes L^{n - 1} + \dots + \ddd g_n \otimes \operatorname{Id}.
$$

\begin{Proposition}\label{p1}
Consider the  operation $\circ$ on $T\mathsf M^n$ defined by  $\xi \circ \eta = T_M(\xi, \eta)$. This operation is commutative and associative.    
\end{Proposition}
\begin{proof}[Sketch of proof]
Lemma 2.4 from \cite{nij4} implies the identity
\begin{equation}\label{sym}
   T_M(L\xi, \eta) = T_M(\xi, L\eta) 
\end{equation}
for all vector fields $\xi, \eta$ and arbitrary symmetry $M$. Let us use $\xi, L\xi, \dots, L^{n - 1}\xi$ for a cyclic basis. Then we get $T_M(L^p\xi, L^q\xi) = R L^{p + q}\xi$, where $R = T(\xi, \cdot)$. Note that $R$ is a  polynomial in $L$ with functional coefficients. Thus, we get $T_M(R\xi, \eta) = T_M(\xi, L\eta)$ from \eqref{sym}. By direct calculations we have
$$
T_M(L^p\xi, T_M(L^q \xi, L^r \xi)) = T_M(L^p\xi, R L^{q + r}\xi) = R^2 L^{p + q + r}\xi = T_M(T_M(L^p\xi, L^q \xi), L^r \xi).
$$
This is exactly the associativity.
\end{proof}

In the definition of $\mathrm F$-manifolds, the important part is vanishing of the so-called Ako-Yano bracket for the commutative associative structure. As we shall see below, this tensor indeed vanishes for the structure given by $T_M$. 

%%%%%%%%%%%
%%%%%%%%%%%
%%%%%%%%%%%
%%%%%%%%%%%
%%%%%%%%%%%

\subsubsection{Application: construction of natural  $\mathrm F$-manifolds via symmetries} \label{natural}

In a neighbourhood of a generic point, the algebra $\mathrm{Nij}\, L$ contains a $\gl$-regular Nijenhuis operator, which in appropriate coordinate system is constant. Assume that $L$ is such an operator. Now pick a regular symmetry $M \in \mathrm{Nij}\, L$. 

\begin{Proposition}\label{p2}
There exists a unique vector field $e$, such that
\begin{enumerate}
    \item $e$ is cyclic for $L$
    \item $\mathcal L_e L = 0$
    \item $\mathcal L_e M = \operatorname{Id}$
\end{enumerate}
\end{Proposition}
\begin{proof}[Sketch of proof]
From Condition (2), we get  $\mathcal L_e M = T_M(e, \cdot)$. Taking $g_i$ from \eqref{decomp} as coordinates, we see that $e$ is uniquely defined. So now we need to prove the existence. Due to the splitting theorem, it is enough to prove it for a Jordan block only.

Now notice that if we pick $g_i$ as coordinates, the operator $L$ is in first companion form. As it can be brought to constant form, the coefficients of its characteristic polynomial are constants (in the case of a $\gl$-regular Nijenhuis operator $L$,  it is equivalent to reducibility of $L$ to a constant form in some coordinates). Thus, $M$ depends linearly on the companion coordinates.

Taking $L$ in Jordan normal form, after some work we can see that $M$ is in Toeplitz form. Then we can take $e$ to be $\partial_n$. It satisfies all the conditions.
\end{proof}

The existence of $e$ implies that $T_M(e, \cdot) = \operatorname{Id}$. This implies that $T_M(L^pe, L^q e) = L^{p + q} e$ and, as shown in \cite{b1}, its Ako-Yano bracket vanishes.

As a result, we get a natural $\mathrm F$-manifold structure around a regular point: 
\begin{enumerate}
    \item tensor $T_M$ defines a commutative associative algebra,
    \item vector field $e$ is its unity,
    \item vector field $M e$ is the Euler vector field.
\end{enumerate}

In particular, the pair $(M, e)$ gives a Nijenhuis structure with unity.

%%%%%%%%%%%
%%%%%%%%%%%
%%%%%%%%%%%
%%%%%%%%%%%
%%%%%%%%%%%

\subsubsection{ Research problems related to symmetries of $\mathrm F$-structures with  unity}

\begin{enumerate}
\item  Let $M$ be a strong symmetry of $L$ and $e$ be a unity. Is it true that $\mathcal L_e M$ is again a strong symmetry of $L$? In the case of a $\gl$-regular Nijenhuis operator around a regular point it is, of course, true.
\item Let $M$ be a symmetry of $L$ and $e$ be a unity. Is it true that  $\mathcal L_e M$ is again a symmetry of $L$? We think it is not, so we need a counterexample.
\item Prove the vanishing of the Ako-Yano bracket in general case, using only tensorial calculations.
\item Is it true that for two different symmetries $M, N$, the operations $T_M, T_N$ are compatible? This is almost obvious and the answer is affirmative.
 
    \item We have constructed an algebra structure which satisfies almost all conditions, except for differential conditions on the unity. Which conditions does the unity satisfy in our case, for example, for a differentially non-degenerate Nijenhuis operator?
\end{enumerate}

%%%%%%%%%%%
%%%%%%%%%%%
%%%%%%%%%%%
%%%%%%%%%%%
%%%%%%%%%%%

\subsubsection{Problems for research:   $\mathrm F$-manifolds}
\begin{enumerate}
    \item Proposition \ref{p2} uses the existence of an operator field $L$ with $\mathcal L_e L = 0$. Is it true that in all known examples, such an operator field exists? It is, of course, non-constant but still.
    \item If the answer to the previous question is affirmative, then we can define an F-type singularity of a Nijenhuis operator as follows: in a neighbourhood of such a point, $\mathrm{Nij}\, L$ contains a Nijenhuis operator, which satisfies $\mathcal L_e L = 0$ (we denote it with the same letter), it is $\gl$-regular and $e$ is cyclic for this operator. Moreover, the structure of the algebra in a neighbourhood of this point is given by the above construction
    \item Are there any left-symmetric algebras that produce such algebras/operators? In other words, is there an intersection of this class of singularities and scalar type singularities. One needs to take a look at the list in dimension three for algebras with functionally independent coefficients of the characteristic polynomial.
\end{enumerate}

%%%%%%%%%%%%
%%%%%%%%%%%%
%%%%%%%%%%%%
%%%%%%%%%%%%
%%%%%%%%%%%%

\subsection{ Frobenius manifolds} \label{sec:Frobenius}

{\bf Definition.}
A pair ($\mathrm F$-structure, 1-form $\alpha$) % better to consider a pair ($F$-stricture, form) than ($F$-stricture, function). They are equivalent only locally. 
is a  {\it  Frobenius structure}, if $\alpha$ satisfies the following  conditions:
\begin{enumerate}
    \item $g( \xi, \nu):= \alpha(\xi \circ \nu)$ is a flat metric (whose connections will be denoted by $\nabla$), 
    \item  $\dd\alpha=0$,
    \item ${\mathcal L}_e g = 0$,
    \item ${\mathcal L}_E \circ = \circ $ \ , \ $ {\mathcal L}_E g= (2- d)g $ (for a certain constant $d$). 
 \end{enumerate}

\vspace{1ex} 

\begin{Theorem}[\cite{HD}, Theorem 8.2.] \label{t6}  For a  regular $\mathrm F$-structure $(L, e)$, there exist  $1$-forms   $\alpha$ such that the corresponding {\rm (}$\mathrm F$-structure,  $1$-form $\alpha${\rm )}  are Frobenius structures. Moreover, the second   jet of $\alpha$ at one point determines the form  $\alpha$. 
\end{Theorem}

Theorem \ref{t6}  implies that   the system of PDEs on $\alpha$ is  a second  order  system in the Cauchy-Frobenius form (all second derivatives are explicitly expressed in terms of the lower derivatives). Note that by the second part  of Property (4), $E$ is a homothety vector field of $g$, so $E_{i,j}=   \nu_{ij} + (2-d) g_{ij}$ for a skew-symmetric 2-form $\nu$.  Actually,   the 2-form   $\nu$ contains the whole information about the first derivatives of the 1-form $\alpha$. 
Therefore,  Conditions 1--4 from the definition of Frobenius manifolds are  a system of PDEs on  the linear vector bundle  which is the direct product $\Lambda_2(T^*\mathsf M)\times T^* \mathsf M$.  

\vspace{1ex}
 Recall that all possible regular  F-structures are constructed (and the freedom is the number of Jordan blocks in the corresponding Nijenhuis operator).  Conditions 2, 3, 4  from the definition of Frobenius manifolds are linear and of second  order in  $F$. 

\vspace{1ex}

A natural approach, from the viewpoint of Nijenhuis geometry, would be (starting from the $\gl$-regular case) to rewrite the system in   a geometric form in terms of $(L,e)$. 

\vspace{1ex}

Theorem \ref{t6}  tells us that regular Frobenius structures  can be viewed as parallel sections of some (most probably, nonlinear) connection. The connection comes   from the F-structure, that is, from  the pair $(L, e)$. One can write them for any pair $(L,e)$. This nonlinear connection is flat, that is, its holonomy group is trivial.

   \begin{Problem} Can one write the system of equations in such a way that its coefficients depend on $L$ only (at least in the $\gl$-regular case)? \end{Problem}
   Note that  in the $\gl$-regular case any two pairs $(L_1,e_1)$ and $(L_2, e_2)$ with identical Jordan blocks of $L_1$ and $L_2$ are isomorphic by a coordinate transformation. Moreover, $\alpha$ is just the dual form for $e$. This  makes us hope that the system can be written in terms of $L$ only.

 \begin{Problem} Are the compatibility conditions for this system precisely the conditions $\mathcal N_L=0$? They must be geometric conditions on $L$ which vanish if $L$ is a Nijenhuis operator. 
  \end{Problem}

The discussion above allows one to reduce the study of regular Frobenius manifolds to the study of pairs $(L,e)$ (plus the choice of $\nu$ and $\alpha$ at one point).

Frobenius manifolds allow us to define solutions of WDVV equation, which is interesting in particular because of their relations to integrable systems.  

 \begin{Problem} Why is the WDVV equation interesting? What are its applications?   Is it possible to reformulate the part which is  interesting for application  by replacing the solution of these equations  with a pair $(L,e)$, at least in the $\gl$-regular case?  What pairs $(L,e)$ appear in known examples of Frobenius manifolds? 
 \end{Problem} 

It is known that the WDVV equation produces a Hamiltonian operator of  order 3 which is not Darboux integrable.  It is known that a Hamiltonian operator of order 3 is essentially the same as a Killing (0,2) tensor  for the flat metric satisfying additional conditions\footnote{See \S \ref{sec:6.1}.}.
The   additional conditions on the tensor is a  nonlinear tensorial expression in the covariant derivatives of the Killing tensor.

 \begin{Problem}  How to write the conditions on this Killing tensor using $L$ or the pair $(L,e)$? If they are fulfilled at one point, are they necessarily fulfilled at all points?  \end{Problem}

\section{Nijenhuis operators: conservation laws and geodesically compatible metrics}

\subsection{Main definitions and basic results} \label{sec:conservation}

Main references are \cite{nij4} and \cite{nijappl5}.
All constructions are local and real analytic.

\begin{definition}
Let $L$ be an operator field (not necessarily Nijenhuis). A function $f$ is said to be a {\it  conservation law }
for $L$, if 
the 1-form $L^*\dd f$ is closed  (in particular, there locally exists a function $g$ such that $\dd g= L^*\dd f$). 
\end{definition}

If $L$ is a Nijenhuis operator,  then 
\begin{itemize}

\item Each conservation law $f=f_1$ generates an hierarchy of conservation laws $f_k$ such that $\dd f_k =L^*\dd f_{k-1} = (L^*)^{k-1} \dd f_1$.

\item If $L=\operatorname{diag}(a_1(u^1), \dots, a_n(u^n))$ with $a_i\ne a_j$  (for $i\ne j$) in a certain coordinate system $u^1, \dots, u^n$, then its conservation laws are all of the form $f=f_1(u^1)  + \dots f_n(u^n)$. 

\item If $L$ is $\gl$-regular and algebraically generic  (i.e.,  locally each eigenvalue has constant mutiplicity) then there is a complete description of conservation laws (see \cite{nij4}).

\item If $L$  is differentially non-degenerate at a point $\mathsf p$ , then every conservation law of $L$ near $\mathsf p$ can be written as $\tr F(L)$, where $F$ is an analytic matrix function (which is well defined locally near  $\mathsf p$). 

\item If $L$ is $\gl$-regular  (but not necessarily algebraically generic so that collisions of eigenvalues are allowed), then locally conservation laws of $L$ are parametrized by $n$ functions of one variable and there always exists a regular conservation law  $f_1$ such that the functions $f_1, \dots, f_n$ from the corresponding hierarchy are functionally independent (i.e. $\dd f_1 \wedge \dd f_2 \wedge \dots \wedge \dd f_n \ne 0$).  In particular, if we take them as local coordinates, then $L$ takes the second companion form
\begin{equation}
\label{eq:Lcomp2}
L_{\mathsf{comp2}} = \begin{pmatrix}
0 & 1 &  & & \\
0 & 0 & 1 & & \\
\vdots & \ddots & \ddots & \ddots & \\
0 & \dots & 0 & 0  & 1\\
\sigma_n & \!\!\!\!\sigma_{n-1} & \dots & \sigma_2 & \sigma_1  \\
\end{pmatrix}
\end{equation}
\end{itemize}

\begin{definition}
A metric $g$ and a $g$-self-adjoint operator $L$ are said to be geodesically compatible\footnote{Explanation of the terminology is as follows: if $g$ and $L$ are geodesically compatible and $L$ is non-degenerate, then $\bar g= \det(L^{-1}) g( L^{-1} \cdot  ,  \cdot  )$ is geodesically equivalent to $g$, that is, every $g$-geodesic, under an appropriate reparameterisation, is a $\bar g$-geodesic.} if
\begin{equation} \label{eq:B1}
    \nabla_\eta L = \frac{1}{2} \Bigl( \eta \otimes \dd \tr L + (\eta \otimes \dd \tr L)^*       \Bigr),
\end{equation}
where $\eta$ is an arbitrary vector field (tangent vector) and $A^*$ denote the $g$-adjoint operator to $A$  (so that in the right hand side we have the $g$-symmetrisation of the operator $\eta \otimes \dd \tr L$). 
\end{definition}

Main properties/facts:

\begin{itemize}

\item  $L$ from the above definition is automatically Nijenhuis  \cite{BM2003}.

\item  $L$ and $g$ are geodesically compatible if and only if $g$ and $\bar g = \frac{1}{|\det L|} g L^{-1}$ 
are geodesically equivalent (i.e. share the same unparametrised geodesics) \cite{BM2003}. This explains the geometric meaning of geodesic compatibility.

\item If $L$ and $g$ are geodesically compatible, then  $L$ and $\tilde g = g f(L)$ are also  geodesically compatible for any function $f$ such that the operator $f(L)$  is defined. Moreover,
$L$ is geodesically compatible with any metric of the form $gM$ where $M$ is a strong symmetry of $L$ \cite{nij4}.

\item  Not every Nijenhuis operator admits a geodesically compatible metric.

\item If $L$ is $\gl$-regular, then there always exists a geodesically compatible metric for it \cite{nijappl5}.

\item If a $\gl$-regular operator $L$ is given in second companion form \eqref{eq:Lcomp2}, then a geodesically compatible metric $g$ can be defined by an explicit formula  \cite{nijappl5}. 

\end{itemize}

\subsection{Research problems}

\begin{Problem}
 Given a Nijenhuis operator $L$, describe all of its (local) conservation laws  (of course, the same question can be asked for any operator, not necessarily Nijenhuis, see below discussion on superintegrable potentials).   
\end{Problem}

\begin{Problem}  Which Nijenhuis operators admit\footnote{The answer is known near almost every points and follows from  the  local description of geodesically compatible pairs in \cite{splitting2, splitting1}. The question is most interesting near singular points of the Nijenhuis operator.} at least one geodesically compatible metric $g$? 
Given $L$, describe all the metrics  geodesically compatible with it.  Under some additional assumptions ($L$ has $n$ different eigenvalues at generic points and  $g$ is Riemannian), the problem was solved in \cite{Matveev2006}. 
The case of metrics of non-Riemannian signatures, even in dimension 2, is still open. The Riemannian case is solved in \cite{BMF}, and the description of 2-dimensional geodesically compatible pairs such that $g$ has signature $(+,-)$ near regular points is in \cite{BMP,japan}.
 \end{Problem} 

More specific tasks are as follows

\begin{Problem}   We have a complete classification of $\gl$-regular Nijenhuis operators in dimension 2 as a list of explicit normal forms.  Describe all conservation laws for each of these operators.  Or at least, find one regular conservation law for each of them.
 \end{Problem} 

Notice that this regular conservation law is expected to be of the form $f_1(\lambda_1) + f_2(\lambda_2)$ where $\lambda_1, \lambda_2$ are the eigenvalues of $L$.  This formula is invariant, but will essentially depend on the type of singularity of $L$.   Recall that we are dealing with singular points at which the eigenvalues collide so that $\lambda_1$ and $\lambda_2$ are not smooth at a given point.  The functions $f_1$ and $f_2$ are not expected to be smooth either, but their combination
$$
f_1(\lambda_1) + f_2(\lambda_2)
$$
must be smooth  (this result is easy to achieve, it is sufficient to take an arbitrary smooth function $f=f_1=f_2$,  then $f(\lambda_1) + f(\lambda_2)$ will be symmetric in $\lambda_1$ and $\lambda_2$ and, therefore, can be expressed as a function $g(\sigma_1, \sigma_2)$ of the coefficients $\sigma_1, \sigma_2$ of the characteristic polynomial which are smooth). In addition, we want $\dd (f_1(\lambda_1) + f_2(\lambda_2))\ne 0$. This condition is tricky, and the above idea does not work as for any smooth function $g(\sigma_1, \sigma_2)$,  we have $\dd g = 0$ because, as a rule, $\dd \sigma_1=\dd\sigma_2=0$ at our singular point. 

Next, if we have a regular conservation law $f$ for $L$, then we should be able to construct a companion coordinate system and then 
get a metric $g$ geodesically compatible with $L$, by using the above mentioned explicit formula from \cite{nijappl5}.

 \begin{Problem} 
 For each $\gl$-regular Nijenhuis operator $L$ in dimension 2, construct (explicitly) a geodesically compatible metric.   \end{Problem} 

 \begin{Problem}  Generalise the above mentioned formula from \cite{nijappl5} to an arbitrary coordinate system $x^1,\dots, x^n$ assuming that in this coordinate system we have a conservation law given as an explicit function $f(x^1,\dots, x^n)$.  \end{Problem} 

 \begin{Problem}   Construct new explicit examples of metrics geodesically compatible with $\gl$-regular operators (at singular points!) in arbitrary dimension.   \end{Problem}

\section{Superintegrable systems and Nijenhuis geometry}

We expect that some ideas from Nijenhuis geometry might be useful in the context of superintegrable systems. This expectation is based on the following three observations:

\begin{itemize}
\item Superintegrable systems are closely related to separation of variables which, in turn, is closely related to Nijenhuis operators  (some people would even say that separation of variables is basically equivalent to the existence of a certain (diagonalisable) Nijenhuis operator)  
\item One of the key properties of a superintegrable system with a Hamiltonian $H = K+V$ (kinetic + potential) and  quadratic integrals $F_i = K_i + V_i$ is that the potential $V$ is a conservation law for  $(1,1)$-Killing tensors related to $K_i$.  In the case of Benenti systems (not necessarily superintegrable), these Killing tensors are closely related to a certain Nijenhuis operator $L$ and the corresponding conservation laws can be found in terms of $L$, see \cite{nij4}. The Killing tensors for superintegrable systems are more complicated, but some ideas may still work.

\item One also may expect an interesting link from superintegrable systems to integrable systems of hydrodynamic type. In the case of diagonalisable systems, the  relation is as follows.  First, by \cite{PSS,Sevennec}, the integrability of a system of hydrodynamic type in the sense of Tsarev, i.e., the existence of symmetries of order one,  is equivalent to the  existence of $n$ functionally independent conservation laws. Second, the existence of a Killing tensor of second order, which can be  simultaneously diagonalised  with the metric, implies the Tsarev integrability and weak nonlinearity \cite{KKM}. 
\end{itemize}

We consider a constant curvature metric  $g$ (it seems that the case of non-zero curvature is simpler).

The problem is to find potentials $V$ such that the Hamiltonian $H =  g^{-1}(p, p) + V(x)$ admits many first integrals of the form
$$
F_i = g^{-1}(K_i p,p) + V_i.
$$ 

The form of $K_i$ is clear.  This is a $(1,1)$-Killing of $g$, since $g$ has constant curvature by our assumptions, 
every $K_i$ is  then a  quadratic forms of Killing vector fields. The only condition to analyse is
$$
\{g^{-1}(p,p), V_i\} = -\{  g^{-1}(K_i p,p), V\},
$$
which is equivalent to the fact that $V$ is a conservation law for $K_i$.  Thus we need to describe those functions $V$ that are conservation laws for many $K_i$'s.    For a given $V$, let $\mathcal K(V)$ denotes the vector space of those $K_i$ for which $V$ is a conservation law.  

\subsection{Questions on Killing tensors and related conservation laws.}

Below is a list of questions/problems/comments which  could clarify the situation.  Some of them, perhaps, are obvious to experts (the authors of this paper do not consider themselves as  experts in this topic).

\begin{itemize}

\item[Q1]  Is it true that in all known examples of superintegrable systems the subspace $\mathcal K(V)$ contains  diagonalisable Killing tensors $K_i$?   What about their   eigenvalues? Can one find such a $K_i$ that all of its eigenvalues  are different?  

\item[Q2]  If so, what kind this $K_i$ is?  Recall that in the case of constant curvature,  
a diagonalisable Killing tensor with different eigenvalues leads to separation of variable. Separation of variables on spaces of constant curvature was described in  \cite{separation}, 
the classifying object is a labeled  in-forest $\mathsf F$. What labeled in-forests  appear within superintegrable systems?

\item[Q3]  If $\mathcal K(V)$ contains a diagonalisable Killing tensor with different eigenvalues, then its conservation laws are all known.  
They are parametrised by $n$ functions of one variable.  In many classical examples these functions are the same (see e.g. formula \eqref{eq:M2} in  \S \ref{sec:2.4.1}).  What happens in the known examples of superintegrable systems? Are these $n$ functions always the same?  

\item[Q4]  Are there any examples with $K_i\in \mathcal K(V)$ being a generic Benenti-Killing tensor  (leading to ellipsoidal coordinates)? Is this a generic case? 

\item[Q5]  Take a generic Killing $K$ of $g$. Does it admit any non-trivial conservation law?

\item[Q6]  If generic Killings admit no conservation laws,  then it is natural to ask for compatibility conditions to the equation
$$
\dd (K^* \dd f)=0.
$$  
Note that in a series of works on algebraic approach to superintegrable systems, e.g. \cite{KSV1, KSV2}, algebraic conditions on the Killing tensor were constructed. Are these conditions just the condition of the existence of sufficiently many conservation laws?

\item[Q7] Again about known examples.  Take $K\in \mathcal K(V)$ which is not Benenti but, on the contrary, is a ``bad'' generic element of $\mathcal K(V)$.  We know that $V$ is a conservation law for $K$.  Does $K$ admit any other conservation laws, or $V$ is unique?    If $K$ admits ``many'' conservation laws, what can we say about them? 

\item[Q8]  Is it true that some quadratic forms of linear Killings turn out to be identically zero  (some sort of Pl\"ucker relations?).  Algebraically these relations can be described as follows.  We consider the momentum mapping
$$
\Phi : T^* S^n  \to so(n+1)^*
$$

Then Killing vector fields  (as functions on $T^* S^n$)  are exactly pullbacks of linear functions $\xi :  so(n+1)^* \to \R$.
Hence Killing tensors of rank 2  (quadratic integrals of $g$) are pullbacks of quadratic functions $F: so(n+1)^* \to \R$, $F=\sum c_{ij} e_i e_j$, where $e_1,\dots, e_n$ is a basis of $so(n)$.  

Among such quadratic functions there are those which vanish identically on the image of the momentum mapping. Namely, the image of $\Phi$ consists of all skew-symmetric matrices $M$ of rank 2.  They can be described by the condition that all diagonal $4\times 4$ minors of $M$ vanish.  These minors are full squares so that the image is defined by means of quadratic relations   $F_{ijkl} :   so(n+1)^* \to \R$ where $i<j<k<l$ determine the corresponding minor.  These relations define trivial quadratic Killings.  And those are the only relations we need.

\item[Q9] 
The question we are interested in can be formulated in more algebraic terms, if we ``transport'' our system to the Lie algebra  $e(n+1)=so(n+1)+\R^{n+1}$. 

The elements of this Lie algebra are pairs $(M, x)$, where $M$ is a skew-symmetric matrix and $x\in \R^{n+1}$ is a column-vector.  The Hamiltonians we need are of the form
$$
H = \tr M^2 + V(x)\quad\mbox{and}\quad  F_i=\tr C_i(M)M + V_i(x),
$$ 
where $C_i: so(n) \to so(n)$ is a symmetric operator.   It would be interesting to see how the collection of operators $C_i$ looks like in known examples.

\item[Q10]  There is a kind of identification of quadratic Killings with abstract curvature tensors. (This interpretation should be related to what was said before:   each symmetric operator $C: so(n+1) \to so(n+1)$ defines a quadratic Killing $K$.   This $C$, as a tensor satisfies all algebraic symmetries of curvature tensor except Bianchi.  It seems that the Bianchi identity will follow from 
 the above discussed relations $F_{ijlk}$).  In any case, it would be interesting to see an interpretation of $\mathcal K(V)$ in terms of abstract curvature tensors.
  
\end{itemize}

\section{Projective invariance of geometric Poisson brackets}

\subsection{Basic definitions} \label{sec:6.1}
We work on the space of real analytic curves  on a real analytic  manifold (see \S \ref{sec:1.1}), and consider the  functionals of the form 
\begin{equation} \label{eq:M8}
    u(x)\mapsto \int \left( \textrm{polynomial in  $u_x$, $u_{xx}$, $\dots $ with coefficients depending on $u(x)$ }\right) \ddd x \, 
\end{equation}
which play the role of a Hamiltonian and commuting integrals in infinite-dimensional integrable systems. 

The polynomial in  $u_x$, $u_{xx}$, $\dots $ with coefficients depending on $u(x)$  will be called  {\it polynomial  density}, we will denote them   by  $\mathcal{H}, \mathcal{F}, \mathcal{G}$ etc. Another common name we will use is {\it differential polynomial,} there is no difference between polynomial  density and differential polynomial.

Let us comment on the interval of integration. For simplicity,  some authors assume that the curves are closed (so $u:S^1 \to M$) and the integral is taken over the whole $S^1$.  The setup imposed later (the so-called ``locality'')   will  imply that the interval of integration does not play any role and only the density matters.  Moreover, densities 
different from each other by a total differential are treated as equivalent,  that is, two densities $\mathcal{F}$ and  $\mathcal{G}$  are equivalent if their difference $\mathcal{F}-\mathcal{G}$ equals $D \mathcal{H}= \tfrac{\ddd}{\ddd x} \mathcal{H}$  for   some density $\mathcal{H}$. A mathematically clean approach can be found\footnote{In dimension $n=1$, the approach is  due to \cite{gdi} and in all dimensions it is known in folklore but we did not find a good source.} in \cite[\S 1.2]{nijapp3}, see also \cite{KV}. Let us give more details. 

Let $\mathfrak{A}$ denote  the algebra of differential polynomials.
 The {\it total $x$-derivative} operation    $D=\tfrac{\ddd}{\ddd x}$  is  defined as follows. One requires  that $D$  satisfies the Leibnitz rule and then defines it on the generators of $\mathfrak{A}$, i.e., on functions $f(u)$   and the $x$-derivatives of $u$  by setting
$$
D(f)= \sum_{i=1}^n \frac{\partial f}{\partial u^i} u^i_x   \quad\mbox{and}\quad     \ D ( u^i_{x^j})= u^i_{x^{j+1}}.
$$ 
Clearly, the operation $D$ increases the differential degree by no more than one.

Next, denote by $\tilde{\mathfrak{A}}$ the quotient algebra $\mathfrak{A}/{D(\mathfrak{A})}$.  The tautological projection $\mathfrak A \to  \tilde{\mathfrak{A}}$ is traditionally denoted by $\mathcal H \mapsto \int \mathcal H \ddd x \in \tilde{\mathfrak{A}}$.
In simple terms it means that we think that two differential polynomials $\mathcal{H}, \bar {\mathcal{H}}$  are  equal, if their difference is a total derivative of a differential polynomial.

Note that by construction, the operation $D$ has the following remarkable property, which  explains its name and  also 
the  	notation  $\tfrac{\ddd}{\ddd x}$ used for $D$. For any curve $c:[a,b]\to U$   and for any  element  $\mathcal{H} \in  {\mathfrak{A}}$ we have: 
\begin{equation}
\label{eq:intro:4}
\tfrac{\ddd}{\ddd x}\left(	\mathcal{H}(\widehat c)\right)= \left(D\mathcal{H}\right)(\widehat c),
\end{equation}
where \begin{equation}
\label{eq:intro:3}
\widehat c:[a,b]\to J^kU\ , \  \     x_0\mapsto \left(c^1,...,c^n,  \tfrac{\ddd}{\ddd x}(c^1),...,\tfrac{\ddd}{\ddd x}(c^n),..., \tfrac{\ddd^k}{\ddd x^k}(c^1),...,\tfrac{\ddd^k}{\ddd x^k}(c^n)\right)_{|{x=x_0}}. 
\end{equation} 
	 		
We will also need another mapping from $\mathfrak{A}$ to the space of $n$-tuples of elements of $\mathfrak{A}$. 
The mapping will be denoted by $\delta$ and will be called {\it the  variational derivative}. 
Its $i^{\textrm{th}}$ component will be denoted by $\tfrac{\delta}{\delta u^i}$ and for an element $\mathcal{H}\in \mathfrak{A}$  it is given by the Euler-Lagrange formula:
$$
 \frac{\delta \mathcal{H}}{\delta u^i}= \sum_{k =0}^\infty (-1)^k D^k\left(\frac{\partial \mathcal{H}}{\partial u_{x^k}^i}\right)
$$
 (only finitely many elements in the sum are different from zero so the result is again a differential polynomial).  It is known, see e.g. \cite{gdi}, that for an element $\mathcal{H}\in \mathfrak{A}$ we have $\delta \mathcal{H} =0$ if and only if $\mathcal H$ is a total $x$-derivative. Then, we see again that the variational derivative does not depend on the choice of a differential polynomial  in the equivalence class of $\mathcal{H}$ in $\mathfrak{A}$.  Thus, the mapping $\delta$ induces a well-defined mapping on $\tilde{\mathfrak{A}}, $ which will be denoted  by the same letter $\delta$.  One can think of $\delta \mathcal{H}$ as a covector  with entries from $\tilde {\mathfrak{A}}$, because the  transformation rule of its entries  under the change of $u$-coordinates is a natural generalisation of  the transformation rule for (0,1)-tensors.

On the space of densities $\tilde{\mathfrak{A}}$  we will consider  the following operation, sending 
the differential polynomials $\mathcal{H}$ of degree $m_1$   and $\mathcal{F}$ of degree $m_2$ to a differential polynomial, denoted by  $\{\mathcal{H}, \mathcal{F}\}$, 
of degree $m_1 +m_2+ k+2$ (where  $k$ will be called the degree of the pairing)  given by 

\begin{equation} \label{eq:M9}
 \{\mathcal{H}, \mathcal{F}\} =   \tfrac{\delta \mathcal{H}}{\delta u^i} \left(  g^{ij} D^k + a^{ij}_p u^p_x D^{k-1} +
 b^{ij}_p u^p_{xx} D^{k-2}+ c^{ij}_{pr}u^p_xu^r_x D^{k-2}+...+ \Gamma^{ij}_su^s_{x^k} \right) \tfrac{\delta \mathcal{F}}{\delta u^j}.     
\end{equation}

The components $g^{ij}, a^{ij}_p,..., \Gamma^{ij}_s$ etc are functions of $u$, one can view them as geometric objects on the manifold. In particular $g^{ij}$ is a $(2,0)$-tensor field. The geometric  nature of the objects $a^{ij}_p, b^{ij}_p, c^{ij}_{pq}$ is more complicated. They are not tensorial objects  but have a more complicated transformation rule under coordinate changes. 

The coefficients  $\Gamma^{ij}_s$  play
a special role. If we lower\footnote{One customary assumes  that $g^{ij}$ is nondegenerate, we will also assume it.} the  first upper index  using $g_{ij}$ and multiply by $-1$,   the obtained object $-\Gamma^i_{jk}$  transforms as Christoffel symbols of an affine connection.

The bilinear operation given by \eqref{eq:M9}  is 
generically   not skew-symmetric and does not satisfy the Jacobi  identity. Both these conditions are in fact conditions on the coefficients  $g^{ij}, a^{ij}_p,..., \Gamma^{ij}_s$.  

\begin{Theorem}[\cite{Doyle}]
    If the  bilinear operation given by \eqref{eq:M9}  is a Poisson bracket, then $-\Gamma^i_{jk}$ is a  symmetric flat connection and $g^{ij}$ is symmetric if $k$ is odd and skew-symmetric if $k$ is even. 
\end{Theorem}

One therefore can work in coordinate systems such that $\Gamma^i_{jk}=0$. Further conditions on the coefficients  $g^{ij}, a^{ij}_p,...,  $ in such a coordinate system lead to more simplifications.

In  the next three examples \ref{ex:M2},\ref{ex:M3} and \ref{ex:M4} we assume that the bilinear operation given by \eqref{eq:M9}  is 
 skew-symmetric and does not satisfy the Jacobi  identity. 

\begin{Ex}[\cite{DN}] \label{ex:M2}
If the degree  $k=1$,  
in the flat coordinate system for the connection $-\Gamma^i_{jk}$,   the components of $g$ are constant (so that $g$ is parallel with respect to  $-\Gamma^i_{jk}$) and \eqref{eq:M9} reads 
\begin{equation} \label{eq:M10}
 \{\mathcal{H}, \mathcal{F}\} =   \tfrac{\delta \mathcal{H}}{\delta u^i}   g^{ij} D   \tfrac{\delta \mathcal{F}}{\delta u^j},
\end{equation}
where $g^{ij}$ is a constant symmetric (nondegenerate) matrix. \end{Ex}

\begin{Ex}[\cite{Po,VV}] \label{ex:M3} 
   If the degree  $k=2$,  
in the flat coordinate system for the connection $-\Gamma^i_{jk}$,   \eqref{eq:M9} reads 
\begin{equation} \label{eq:M11}
 \{\mathcal{H}, \mathcal{F}\} =   \tfrac{\delta \mathcal{H}}{\delta u^i}   Dg^{ij} D   \tfrac{\delta \mathcal{F}}{\delta u^j},
\end{equation}
where $g^{ij}$ is the inverse matrix to a Killing-Yano tensor (with respect to the connection  $-\Gamma^i_{jk}$).  \end{Ex} 
 
 Recall that a skew-symmetric   $(0,2)$-tensor  $g_{ij}$ is Killing-Yano, if its covariant derivative is skew-symmetric. In the index notation, the Killing-Yano equation is 
 \begin{equation} \label{eq:M14}
    g_{i(j,k)}=0,  
 \end{equation}
 where the parentheses denote the symmetrisation. 
 Another equivalent way to write down this condition is as follows: 
 the covariant derivative  of $g_{ij}$ coincides with the exterior derivative.

\begin{Ex}[\cite{FPV1, FPV2}] \label{ex:M4}
   If the degree  $k=3$, the inverse matrix to $g^{ij}$ is a Killing  (0,2) tensor with respect to the connection $-\Gamma_{jk}^i$.
In the flat coordinate system for the connection $-\Gamma^i_{jk}$,  the equation \eqref{eq:M9} reads 
\begin{equation} \label{eq:M12}
 \{\mathcal{H}, \mathcal{F}\} =   \tfrac{\delta \mathcal{H}}{\delta u^i}   D\left(g^{ij}  + c^{ij}_r u^r_x    \right)  D   \tfrac{\delta \mathcal{F}}{\delta u^j}.
\end{equation}
The components $c^{ij}_s$ are essentially covariant derivatives of $g_{ij}$ in the connection  $-\Gamma^i_{jk}$. Moreover, the tensors  $g$ and $ c$  satisfy a certain nonlinear algebraic condition. 
\end{Ex}

\subsection{ Weigthed tensor fields, projectely invariant equations, and 
projective invariance of Poisson  brackets for $k=2,3$}

We first recall the notion of a projective structure.  
A slightly  informal and ineffective definition is as follows: {\it projective structure} on an $n$-dimensional manifold $M$ is a smooth family $\mathcal{F}$ of smooth curves   such that   \begin{itemize}
 \item at any point and in any direction there exists precisely one curve from this family passing through this point in this direction, and   
\item there exists an affine connection $\nabla=(\Gamma^i_{jk})$
 such that each curve from this family, after a proper reparameterisation, is a geodesic of this connection. 
 \end{itemize}

 Note that the equation of geodesics of the connection $\Gamma^i_{jk}$
  and  of its symmetrization $\tilde \Gamma_{jk}^i= \tfrac{1}{2} \left(\Gamma^i_{jk} + \Gamma^i_{kj}\right)$ are clearly the same, since the connection comes symmetrically in the defining equation 
  \begin{equation} \label{geodesic} 
  \ddot \gamma^i + \Gamma_{jk}^i\dot \gamma^k \dot \gamma^j=0\end{equation}  
  of a geodesic; without loss of generality we will 
   therefore always assume that the connections we consider  are torsion-free. 
 
 A simplest example of a projective  structure is the family $\mathcal{F}$ consisting of all straight lines.  Such a projective structure is called {\it flat}. A slightly more complicated example is when we pick 
 any connection $\Gamma^i_{jk}$ and put $\mathcal{F}$ to be  all geodesics of this connection.  Since  there is (up to a reparameterisation) a unique geodesic of a given connection passing through a given point and tangent to a given direction, the second example suggests how one can provide a description of all projective structures: one needs to understand what connections  have the same geodesics viewed as unparameterized curves.  
 This understanding is provided by the following theorem, which was proved at least in   \cite{Levi-Civita1896}; we give the answer in the notation of  \cite{Weyl1921}.   
 
 We call   connections having 
  the same geodesics viewed as unparameterized curves \emph{geodesically (or projectively) equivalent}.

\begin{Theorem}[ \cite{Levi-Civita1896,Weyl1921}]  \label{thm:1}  $\nabla = (\Gamma_{jk}^i)$ is geodesically  equivalent to 
 $ \bar \nabla = (\bar{\Gamma}_{jk}^i)$, if and only if there exists an $1$-form $\phi= \phi_i$ such that 
\begin{equation} 
\bar  \Gamma^i_{jk}= \Gamma_{jk}^i+ \phi_k \delta^i_j + \phi_j \delta^i_k.   \label{ast}
 \end{equation} \end{Theorem}

In the index-free notation, the equation \eqref{ast} reads: 

\begin{equation} \label{ast2}
\bar \nabla_X Y=\nabla_X Y +  \phi(Y)X  + \phi(X)Y. 
 \end{equation}

We see that  the condition that two connections are geodesically  equivalent is quite  flexible:  for a given connection the set  of 
 connections that are geodesically  equivalent to it is an  infinitely-dimensional affine subspace in the affine space of all connections. The freedom in choosing such a connection is the freedom of choosing a 1-form $\phi$  in \eqref{ast}. 
In particular,  two geodesically equivalent connections can coincide at some open nonempty set  and be different on another open nonempty subset.

 This flexibility fails if we work not with connections but with metrics. The condition that a metric  has a  nonproportional   geodesically equivalent one is a very strong condition on the metric. Indeed, by \cite{KM2016,Matveev2012a}, a generic metric does not admit any nonproportional geodesically equivalent partner. 

In particular, a metric geodesically equivalent to the flat projective structure has  a constant curvature: 
\begin{Theorem}[Beltrami Theorem; \cite{Beltrami1865} for dim $2$; \cite{Schur1886} for dim$>2$] 
A metric geodesically  equivalent to a metric of constant curvature has constant curvature. 
\end{Theorem}

An important concept which naturally comes in the projective differential geometry and will be useful for us in the consequent is the concept of weighted tensors. We will recall them  following \cite{Eastwood, Matveev2018}.  Note, however,  that the papers \cite{Eastwood} and  \cite{Matveev2018} use different conventions: the weight $k$ in \cite{Eastwood} corresponds to the weight $-k$ in \cite{Matveev2018}. Below, we  follow the convention of  \cite{Eastwood}.

We assume that our manifold  $M$ is orientable (or we work locally) and fix an orientation. The dimension  $n$ is assumed to be $\ge 2$. We consider the bundle 
 $ \Lambda_nM $ of positive volume forms on $M$. 
 Recall that locally  a volume form is a skew-symmetic 
form of maximal order, in local coordinates $x=(x^1,...,x^n)$ 
one  can always write it as  
$f(x) dx^1 \wedge ... \wedge dx^n$ with $f\ne 0$.  The word ``positive''   means that  if the basis 
$\tfrac{\partial }{\partial x^1},..., \tfrac{\partial }{\partial x^n}$ is positively oriented, which we will always assume later,  then $f(x)>0$.

\vspace{1ex}

Positive volume forms are  naturally  organised  in    a locally trivial   1-dimensional  bundle over our manifold $M$ with the structure group $(\mathbb{R}_{>0}, \cdot)$.  
Let us discuss two natural ways for  a local trivialization of this bundle:

\begin{enumerate} \item  Choose a section in this bundle, i.e., a positive  volume form $\Omega_0=f_0  dx^1 \wedge ... \wedge dx^n$ with $f_0\ne 0$. 
Then the other sections of this bundle can be thought to be positive functions on the manifold: the form $f   dx^1 \wedge ... \wedge dx^n$  is then essentially the same as the function $\tfrac{f}{f_0}$.  In particular,   if we change coordinates, the ratio $\tfrac{f}{f_0}$   transforms like a  function, since both coefficients, $f$ and $f_0$, are multiplied by the determinat of the Jacobi matrix.   This way to trivialize the bundle of the volume forms  will be actively used later,  and will be very effective when the volume form $f_0  dx^1 \wedge ... \wedge dx^n$ with $f\ne 0$ is parallel with respect to some preferred   affine connection in the projective class. Note that this trivialization is actually a global one (provided the volume form $\Omega_0$ 
 is defined on the whole manifold).

 \item In local coordinates $x=(x^1,...,x^n)$, we may think that   the volume form 
 $\Omega = f(x)  dx^1\wedge ...\wedge dx^n $ corresponds to the local  function $f(x)$.   In this case, $f(x)$ cannot be viewed as a function on the manifold since  
 its transformation rule is different from that of functions:   a coordinate change  
 $ x= x(y)$  transforms  $f(x)$ to $\det\left(J \right) f(x(y))$, where $  J= \left(\tfrac{dx}{dy}\right)$  is the Jacobi matrix.   
Note that this way to trivialize the bundle can be viewed as  a special case of the previous way, with the form  $ \Omega_0=   dx^1 \wedge ... \wedge dx^n$; though of course this  form $\Omega_0$ depends on the choice of local coordinates.  \end{enumerate}

   Now, take  
    $\alpha \in \mathbb{R}\setminus \{0\}$. Since $t \to t^\alpha$ is an isomorphism of $(\mathbb{R}_{>0}, \cdot)$, 
for any 1-dimensional $(\mathbb{R}_{>0}, \cdot)$-bundle,  its power $\alpha$ is well-defined and is also a one-dimensional bundle. We consider 
$\left(\Lambda_n\right)^\alpha M$. It is an 1-dimensional bundle, so its sections locally can be viewed as functions. Again we have two ways to view the sections as functions:

\begin{enumerate} \item[(A)]  Choose   a volume form $\Omega_0=f_0 dx^1\wedge ... \wedge dx^n$,
 and the corresponding  section $\omega = (\Omega)^\alpha$ of
$ \left(\Lambda_n\right)^\alpha M$. Then  the other sections of this bundle can be thought to be positive functions on the manifold.

 \item[(B)] In local coordinates $x=(x^1,...,x^n)$, we can choose the section  $(dx^1\wedge ...\wedge dx^n)^\alpha$, then the section  
 $\omega = (f(x)  dx^1\wedge ...\wedge dx^n)^\alpha $ corresponds locally to the function $(f(x))^\alpha$. Its transformation rule is different from that of functions:  a coordinate change  
 $ x= x(y)$  transforms  $(f(x))^\alpha$ to $\left(\det\left(\tfrac{dx}{dy}\right)\right)^\alpha  f(x(y))^\alpha$.   
 
 \end{enumerate}

 By a   {$(p,q)$-{\it tensor field of projective weight $k$}} we understand a section of the following bundle:

$$
T^{(p,q)}M\otimes \left(\Lambda_n\right)^{\tfrac{-k}{n+1}} M    \ \ (  \textrm{notation} \  :=  T^{(p,q)}M(k))
$$

\vspace{1ex} 
  If we have a preferred volume form on the manifold, the sections of  $T^{(p,q)}M(k)$  can be identified with 
$(p,q)$-tensor fields. The identification depends of course on the choice of the volume form.    Actually, if the chosen   volume form is parallel w.r.t. to a connection, then even the formula for the 
covariant derivative of this section coincides with that   for tensor fields.

\vspace{1ex} 
     If we have no preferred volume form on the manifold, in a local coordinate system one can choose  $(dx^1\wedge ...\wedge dx^n) $ as the preferred volume, and still think that sections are    ``almost'' $(p,q)$-tensors:  in  local coordinates, they are also given by $n^{p+q}$ functions,  
  but their transformation rule is slightly different from that for tensors: in addition  to the usual transformation rule for tensors one needs to multiply  the result by 
$\left(\det\left(\tfrac{dx}{dy}\right)\right)^\alpha  $ with $\alpha =\tfrac{-k}{n+1}$.   In particular, the formula for Lie derivative is different from that for tensors.  Also the formula for the covariant derivative is different from that for tensors: one needs to take into account the covariant derivative of $(dx^1\wedge ...\wedge dx^n) $. We are going to discuss this right now.

The  bundle of $(p,q)$-weighted tensors of weight $\alpha$  is an associated  bundle to the tangent bundle, 
so a connection $\left(\Gamma^{i}_{jk}\right)$ induces a covariant derivation on it. The next proposition shows how the covariant derivation transforms if we replace the  connection $\left(\Gamma^{i}_{jk}\right)$ by a geodesically  equivalent connection.

\begin{Proposition}  Suppose (geodesically equivalent) connections $\nabla = (\Gamma_{jk}^i)$ and $ \bar \nabla = (\Gamma_{jk}^i)$  are related by \eqref{ast2}.  
  Then the covariant derivatives  of a volume form $\Omega\in \Gamma \left(\Lambda_n M\right) $ in the connections $\nabla $ and $\bar \nabla$ are related by   
 
\begin{equation} \label{cov1}
 \bar \nabla_X \Omega= \nabla_X \Omega - (n+1) \phi(X) \Omega. 
\end{equation}

   In particular, the covariant derivatives of the    section    $ \omega :=  \Omega^{\tfrac{-k}{n+1}} \in \Gamma(\left(\Lambda_n\right)^{\tfrac{k}{n+1}} M)$
 are related by 
\begin{equation} \label{cov2} 
 \bar \nabla_X \omega = \nabla_X \omega + k  \phi(X) \omega.   
\end{equation} 
\end{Proposition}

The proof of the proposition is straightforward: the proof of \eqref{cov1} can be done by brute force calculations, 
and  \eqref{cov2} follows from \eqref{cov1} and the Leibniz rule.

\begin{definition} We say that a tensorial\footnote{That is, given by tensor operations like contractions, (anti)symmetrisations, tensor products  and covariant derivatives.} 
PDE on   weighted tensors  is projectively invariant, if it does not change when we replace the connection by a geodesically equivalent one. 
\end{definition}

Let us give examples: 

\begin{Ex}[\cite{EM2007}, Theorem 2.5 of \cite{Matveev2018}]\label{ex:M5}     For symmetric $(2,0)$-tensors   $\sigma$  of projective weight $-2$ the operation 
\begin{equation} \label{E4}
\sigma^{ij} \mapsto \sigma^{ij}_{ \ \ , k} - \tfrac{1}{n+1}   ( \sigma^{is }_{ \ \ , s} \delta^j_k +  \sigma^{js }_{ \ \ , s} \delta^i_k)  
\end{equation} 
 is projectively invariant.   In particular, the equation 
 \begin{equation} \label{eq:M13} \sigma^{ij}_{ \ \ , k} - \tfrac{1}{n+1}   ( \sigma^{is }_{ \ \ , s} \delta^j_k +  \sigma^{js }_{ \ \ , s} \delta^i_k)  =0 \end{equation}
 is projectively invariant: if we replace the connection by a geodesically equivalent connection, the equation does not change. 
 \end{Ex}

The equation \eqref{eq:M13}   has the following geometric meaning: its  nondegenerate solutions (for a given connection) 
are essentially the same as metrics  geodesically equivalent to the connection. Namely, the following holds: 

\begin{itemize}  
\item  Suppose  a metric $g_{ij}$ is  geodesically equivalent to the connection used in \eqref{E4}.  Then,  the   $\sigma^{ij}$ given by  the matrix 
\begin{equation} \label{eq:sigma1}
 \sigma^{ij} :=  \left(g^{ij} \otimes \left(  \operatorname{Vol}_g\right)^{\tfrac{2}{n+1}}\right)= g^{ij} |\det g|^{\tfrac{1}{n+1}}
\end{equation}
   is a  solution of \eqref{E4}.  

\item For  a solution  $\sigma= \sigma^{ij}$ of the metrisability  equation  \eqref{E4} such that its determinant in not zero, the corresponding metric geodesically equivalent to the connection  used in \eqref{E4}  is given by 
\begin{equation} \label{eq:sigma}
 g^{ij} := |\det(\sigma)| \sigma^{ij}. 
\end{equation}
\end{itemize}

\begin{Ex}[Folklore]\label{ex:M6}     For skew-symmetric $(0,2)$-tensors   $\sigma_{ij}$  of projective weight $3$ the operation 
\begin{equation} \label{E5-1}
 \sigma_{ij}\mapsto \sigma_{i(j  , k)} 
\end{equation} 
 is projectively invariant.  \end{Ex} 

This of course implies that the equation 
\begin{equation} \label{E5}
 \sigma_{i(j  , k)}=0 
\end{equation} is projectively invariant. 
The equation  has a clear geometric meaning: it is the Killing-Yano equation, see  \eqref{eq:M14}.

\begin{Ex}[Corollary 2.4 of \cite{Matveev2018}]\label{ex:M7}     For  symmetric $(0,2)$-tensors   $\sigma_{ij}$  of projective weight $4$ the operation 
\begin{equation} \label{E6-1}
 \sigma_{ij}\mapsto \sigma_{(ij  , k)} 
\end{equation} 
 is projectively invariant.
\end{Ex}
Therefore, the equation 
\begin{equation} \label{E6}
 \sigma_{(ij  , k)}=0 
\end{equation} 
is projectively invariant. This equation  has a clear geometric meaning: it is the so-called Killing
equation and its solutions are Killing tensors.

Let us compare Examples \ref{ex:M2} and \ref{ex:M5}, Examples \ref{ex:M3} and \ref{ex:M6},  and Examples  \ref{ex:M4} and \ref{ex:M7}. We see that the geometric object inside is the same: 
metric in Examples \ref{ex:M2} and \ref{ex:M5}, Killing-Yano tensor in Examples \ref{ex:M3} and \ref{ex:M6},  and Killing tensor in Examples \ref{ex:M2} and \ref{ex:M5}. This may imply that  geometric Poisson bracket are 
 projectively invariant, and it is indeed the case for geometric Poisson brackets of degree $2$ and $3$.
 For the simplest case $k=1$, it is known in the folklore that the projective invariance fails,  though the object inside, the tensor $g^{ij}$,  is a projectively invariant object. Indeed, it corresponds to solutions of projectively invariant metrisation equation. 
 For $k=2, $ the projective invariance of geometric Poisson brackets was  shown in \cite{VV}. For $k=3$ it was shown in \cite{FPV1, FPV2}. 

Note also that the additional condition on the Killing tensor, 
the nonlinear condition on $c_{ijk}$ mentioned in Example \ref{ex:M4}, 
is projectively invariant, it does not depend on the choice of geodesically equivalent connection.

\subsection{Questions related to interplay of projective differential geometry and geometric Poisson structures } 

\begin{Problem}
  Are geometric  Poisson  structures of any   order  greater than 1 projectively invariant?  Can one interpret the leading coefficient $g^{ij}$ as a projectively invariant object? 
\end{Problem}

Note that the weight of $g^{ij}$ must necessarily be $-(k+1)$. 

The projective invariance of an equation is a very strong condition, most covariant equations are not  projectively invariant. Would geometric Poisson structures or at least the corresponding  $g^{ij}$  be projectively invariant, one could use the developed machinery of parabolic Cartan geometry in the theory of geometric Poisson brackets. 

Note also that the metrisability, Killing-Yano and Killing   equations discussed above are in a certain sense simplest projectively invariant equations;  in  the projective differential geometry they  are linked to  the so-called BGG machinery.  
One may hope that a projectively invariant equation on a skew-symmetric tensors $g_{ij}$  of 
projective weight $ 5 $ (or on symmetric tensors   $g^{ij}$ with projective weight $-5$) 
will provide   new examples of geometric Poisson brackets of degree 4. Note  though that  unfortunately there is no linear projectively invariant equations on skew-symmetric tensors  $g_{ij}$ of   weight $5$.  But we do not know any evident  
reason  why the linearity should be fulfilled. 
Of course,  the following analogy to the finite-dimensional Poisson brackets may hold: In the nondegenerate case, Poisson bracket and symplectic form are essentially equivalent objects, as one is given by the inverse  matrix of the other. The differential geometric condition on the skew-symmetric (2,0) tensor $P^{ij}$ corresponding to the Jacobi identity is nonlinear:
$$
P^{k\ell}\tfrac{\partial }{\partial x^\ell}P^{ij} + P^{i\ell}\tfrac{\partial }{\partial x^\ell}P^{jk} + P^{j\ell}\tfrac{\partial }{\partial x^\ell}P^{ki}=0. 
$$
The differential geometric condition on the skew-symmetric $(0,2)$ tensor $\omega_{ij}$ corresponding to the closedness  is linear.  But we were not able to extend this analogy to an infinite-dimensional case. On  the other side, the projective changes of the connection we consider are  very special, as the projectively equivalent connection $ \bar \Gamma^i_{jk}$ is also flat, that is, its curvature is zero.

For the geometric infinite-dimensional Poisson brackets we have a similar phenomenon for  $k=2$ and $k=3$ (since the Killing equation and the Killing-Yano equations are linear). This visual similarity may be  a coincidence, since it fails for $k=1$ as the metrisability equations viewed as equations on weighted (0,2) tensor are not  linear. Moreover, the Jacobi identity  implies a nonlinear  
condition on the conponents $c^{ij}_k$  from \eqref{eq:M12}.  

\begin{Problem} \label{p11} 
Find  differential condition on  skew-symmetric tensors  $g_{ij}$ of weight 5 which is invariant  under projective change of the connection leaving the connection flat.    
\end{Problem}

Note that in the case $k>3$, the  description of geometric Poisson brackets is  not known and  sometimes is considered to be out of reach.  There are merely a very short list of examples and a conjecture of 
 Mokhov  which can be equivalently reformulated as a projectively invariant form of all geometric homogeneous Hamiltonian structures of order $\ge 2$.   A solution of Problem \ref{p11} can lead to fist nontrivial  examples with $k=4$.  

Note also that if a geometric Poisson structure of order $k$ were projectively invariant, then the weight of the 
 corresponding  $g^{ij}$  would necessarily be equal to $-(k+1)$.

\begin{Problem}
   Certain Frobenius manifolds seem to give   examples  of geometric Poisson bracket of order 3 
   (at least, the WDVV equation does it). Is Frobenius structure projectively invariant? Can one replace, in the definition of Frobenius structure, the flat metric by a metric of constant curvature? 
\end{Problem}

\begin{Problem}
  As we mentioned above, geometric Poisson brackets of order 3  are essentially the same as Killing  tensors whose covariant derivative satisfies a certain relation, and certain Forbenius manifolds were used to construct such Killing tensors.  Does     there exists a construction of a Killing tensor by a Frobenius manifold? 
\end{Problem}

\begin{Problem}
Can two   different  Poisson structures of order 3, one of which is not Darboux-Poisson,  be    compatible?  
\end{Problem}

Let us recall the notion  of `Darboux-Poisson structure' introduced in \cite{Doyle}.  The Poisson structure \eqref{eq:M12}  of order 3 is {\it Darboux-Poisson} if   the coefficients  $c^{ij}_p$ are zero. The latter condition also  implies that 
the coefficients   $g^{ij}$ are constants.

Compatible Darboux-Poisson structures of order 3 were studied in \cite{nijapp3}.  It was shown there
 that 
  the connections $\Gamma^i_{jk}$ corresponding to these two Poisson structures coincide so the corresponding 
	metrics $g_{ij}$  are affinely equivalent. In particular, the operator connecting these two metrics is  Nijenhuis.   
	
	We hope that Nijenhuis operators will appear in a similar manner also in the general, not necessarily Darboux-Poisson case, and their studies will help to understand the situation. Our expectation is that two Poisson structures of order 3 which are not Darboux-Poisson  cannot be compatible.

\begin{Problem} \label{prob:6.5} 
Describe all geometric Poisson structures of the first order compatible with a geometric Poisson structure of the third order which is not Darboux-Poisson. How big can the compatible pencil of such first order  structures be? 
\end{Problem} 

Since geometric Poisson structures of the third order are projectively invariant, all other objects which are compatible with them in any geometric sense should behave nicely with respect to projective transformations. We suggest to employ this in the investigations related to Problem \ref{prob:6.5}.

\subsection*{Acknowledgements.} V.M. thanks the DFG (projects 455806247 and 529233771), and the ARC Discovery Programme DP210100951 for their support. A part of the work was done during the preworkshop and workshop on Nijenhuis Geometry and Integrable Systems at La Trobe University and the MATRIX Institute. The participation of A.K. and V.M. in the workshop was supported by the Simons Foundation, and the participation of A.K. in the preworkshop was partially supported by the ARC Discovery Programme DP210100951. V.M.  thanks the Sydney Mathematics  Research Institute for hospitality and for partial financial support during his visits in  2023 and 2024.

\printbibliography

\end{document}